\documentclass[10pt, notitlepage]{amsart}
\usepackage[utf8]{inputenc}
\usepackage[english]{babel}
\usepackage{setspace}
\usepackage{tikz-cd}
\usepackage{amsmath}
\usepackage{amssymb}
\usepackage{amsthm}
\usepackage{amsfonts}
\usepackage{bbm}
\usepackage{float}
\usepackage{braket}
\usepackage{multicol}
\usepackage{mathtools}
\usepackage{url}
\usepackage{hyperref}
\usepackage{enumitem}

\usetikzlibrary{decorations.pathmorphing}

\theoremstyle{definition}
\newtheorem{theorem}{Theorem}[section]
\newtheorem*{theorem*}{Theorem}
\newtheorem{ltheorem}{Theorem}[section]

\newtheorem{lcorollary}[ltheorem]{Corollary}

\newtheorem{llemma}[ltheorem]{Lemma}

\newtheorem{corollary}[theorem]{Corollary}
\newtheorem{lemma}[theorem]{Lemma}
\newtheorem{definition}[theorem]{Definition}

\newtheorem{proposition}[theorem]{Proposition}
\newtheorem{proposition-definition}[theorem]{Proposition-Definition}
\newtheorem{remark}[theorem]{Remark}

\newtheorem*{conjecture*}{Conjecture}

\makeatletter
\renewcommand\tableofcontents{%
	\null\hfill\textbf{\Large\contentsname}\hfill\null\par
	\@mkboth{\MakeUppercase\contentsname}{\MakeUppercase\contentsname}%
	\@starttoc{toc}%
}
\makeatother

\title[Bijective rigidity and injectivity of the comparison map]{Bijective rigidity of uniform Roe algebras and injectivity of the comparison map}
\author{K. Krutoy}
\date{\today}
\email{krutoy@imj-prg.fr}
\address{Institut de Math\'ematiques de Jussieu - Paris Rive Gauche (IMJ-PRG) \\ Universit\'e Paris Cit\'e, B\^atiment Sophie Germain, 8 Place Aur\'elie Nemours, 75013 Paris, France}

\begin{document}
	\maketitle \vspace{-20pt}
	\begin{abstract}
		We show that, for uniformly locally finite metric spaces $X$ and $Y$ with isomorphic uniform Roe algebras $C^*_u(X)$ and $C^*_u(Y)$, the existence of a bijective coarse equivalence $f \colon X \to Y$ is equivalent to the injectivity of the $0$th comparison map appearing in the HK conjecture for coarse groupoids.

		We further prove that the $0$th comparison map is injective unconditionally. Moreover, if the underlying space is coarsely connected, this map is in fact split-injective.
	\end{abstract}

	\section*{Introduction}
	\setcounter{section}{0}
	Coarse geometry studies metric spaces by focusing on their large-scale structure while disregarding small-scale features. Since local properties play no essential role in this setting, it is customary to restrict attention to discrete metric spaces. Let $(X,d)$ and $(Y,\partial)$ be metric spaces. A map $f \colon X \to Y$ is called \emph{coarse} if for every $R \ge 0$ there exists $S \ge 0$ such that $f$ maps subsets of $X$ of diameter at most $R$ to subsets of $Y$ of diameter at most $S$. Coarse maps are precisely those maps that preserve the large-scale geometry of metric spaces. Two coarse maps $f,g \colon X \to Y$ are said to be \emph{close} if they are at bounded distance, that is,
	\[
		\sup_{x \in X} \partial\bigl(f(x), g(x)\bigr) < \infty.
	\]
	Throughout this paper, we work with \emph{uniformly locally finite} metric spaces (see Definition~\ref{definition: ulf space}), also referred to as \emph{bounded geometry metric spaces}. A classical example is provided by a Cayley graph of a finitely generated group with the shortest path metric. The \emph{coarse category} $\mathbf{Coarse}$ is the category whose objects are uniformly locally finite metric spaces and whose morphisms are closeness classes of uniformly finite-to-one\footnote{ 		
	A map $f \colon X \to Y$ is \emph{uniformly finite-to-one} if there exists $N \ge 1$ such that the preimage of every point of $Y$ consists of at most $N$ elements (see Section~\ref{subsection: coarse geometry}). Coarse equivalences are automatically uniformly finite-to-one.}				
	coarse maps. Isomorphisms in $\mathbf{Coarse}$ are called \emph{coarse equivalences}. In the case of finitely generated groups equipped with word metrics, the notion of coarse equivalence coincides with quasi-isometry. We say that two uniformly locally finite metric spaces $(X,d)$ and $(Y,\partial)$ are \emph{(bijectively) coarsely equivalent} if there exists a (bijective) coarse equivalence between them.
	
	Coarse geometry is intrinsically connected to the theory of étale groupoids\footnote{								
			We refer the reader to~\cite{sims2017etale} for a detailed introduction to étale groupoids and their $C^*$-algebras.
						}.	
	An étale groupoid is a topological groupoid $\mathcal{G}$ with unit space $\mathcal{G}^{(0)}$ such that the source and range maps $s,r \colon \mathcal{G} \to \mathcal{G}^{(0)}$ are local homeomorphisms. An étale groupoid $\mathcal{G}$ is said to be \emph{ample} if its unit space $\mathcal{G}^{(0)}$ is totally disconnected. The theory of étale groupoids has attracted significant attention due to its deep connections with topological dynamics, number theory, and, in particular, operator algebras. To any étale groupoid $\mathcal{G}$ one can associate a $C^*$-algebra $C^*_r(\mathcal{G})$, called the \emph{reduced groupoid $C^*$-algebra} of $\mathcal{G}$. In~\cite{li2020every}, Li showed that every classifiable\footnote{							
		By a ``classifiable $C^*$-algebra'' we mean a unital, separable $C^*$-algebra with finite nuclear dimension that satisfies UCT.
	} 
	simple $C^*$-algebra arises as the reduced groupoid $C^*$-algebra of a topologically principal (possibly twisted) étale groupoid. In~\cite{crainic1999AHT}, Crainic and Moerdijk introduced a homology theory for étale groupoids, which is now a substantial tool for analysing the $K$-theory of the reduced groupoid $C^*$-algebras arising from ample groupoids. Precisely, in~\cite{matui2016etale} Matui showed that, for groupoids arising from shifts of finite type, the $K_0$-groups (respectively, $K_1$-groups) of the associated reduced groupoid $C^*$-algebras are isomorphic to the direct sum of the even (respectively, odd) homology groups. He conjectured that this phenomenon extends to more general ample groupoids.

	\begin{conjecture*}[HK conjecture]\label{conjecture: HK}
		Let $\mathcal{G}$ be a minimal, essentially principal, ample groupoid with compact unit space. For $* = 0,1$, there is an isomorphism
		\begin{equation*}
			K_*\bigl(C^*_r(\mathcal{G})\bigr) \cong \bigoplus_{i = 0}^{\infty} H_{2i + *}(\mathcal{G}; \mathbb{Z}),
		\end{equation*}
		where $H_{*}(\mathcal{G}; \mathbb{Z})$ are the Crainic and Moerdijk homology groups of $\mathcal{G}$.
	\end{conjecture*}

	\noindent The HK conjecture has been verified for many classes of ample groupoids, in some cases even beyond the stated hypotheses. However, counterexamples to the HK conjecture were constructed by Scarparo~\cite{scarparo2020homology} and Deeley~\cite{deeley2023counterexample}. It is well known that for any ample groupoid $\mathcal{G}$ and $i = 0,1$ there exists canonical group homomorphism
	$$
		\alpha_i \colon H_i(\mathcal{G}; \mathbb{Z}) \longrightarrow K_i\bigl(C^*_r(\mathcal{G})\bigr),
	$$
	called the \emph{$i$th comparison map} (see Definition~\ref{definition: comparison map} for the definition of $\alpha_0$ in the case of coarse groupoids). At present, only partial results are known concerning the existence and properties of comparison maps in higher degrees. 
	
	For a uniformly locally finite metric space $(X,d)$, the authors of \cite{skandalis2002coarse} constructed a Hausdorff, principal, ample groupoid $\mathcal{G}(X,d)$, called the \emph{coarse groupoid} of $(X,d)$, whose unit space is the Stone-Cech compactification of $X$. The construction of the coarse groupoid yields a functor from the coarse category to the category of étale groupoids with generalised groupoid homomorphisms. The authors of~\cite{bonicke2023dynamic} proved that the HK-conjecture holds for coarse groupoids associated with uniformly locally finite metric spaces of sufficiently small asymptotic dimension. Moreover, they showed that the groupoid homology of the coarse groupoid $\mathcal{G}(X,d)$ is canonically isomorphic to the uniformly finite homology (see Definition \ref{definition: uf-homology}) of $(X,d)$, which will be denoted by $H_i^{\mathrm{uf}}(X; \mathbb{Z})$. The reduced groupoid $C^*$-algebra of a coarse groupoid associated to a uniformly locally finite metric space $(X,d)$ is more commonly known as the \emph{uniform Roe algebra} $C^*_u(X)$. The latter has a simple description in terms of bounded operators on the Hilbert space $\ell^2(X)$. A bounded operator $T \in \mathcal{B}(\ell^2(X))$ is said to have \emph{controlled propagation} if the quantity
	\[
		\operatorname{prop}(T) = \sup \bigl\{ 
			d(x,y) \,\bigm|\, (x,y) \in X \times X  \text{ such that }\langle T\delta_x, \delta_y \rangle \neq 0 \bigr\}
	\]
	is finite. The \emph{uniform Roe algebra} $C^*_u(X)$ is defined as the norm closure in $\mathcal{B}(\ell^2(X))$ of the $*$-algebra of bounded operators with controlled propagation. If $G$ is a finitely generated discrete group equipped with a word metric, then the uniform Roe algebra of $G$ is canonically isomorphic to the reduced crossed product $\ell^\infty(G) \rtimes_r G$. We refer to \cite[Proposition 10.29]{roe2003lectures} for an explicit construction of the isomorphism between $C^*_u(X)$ and $C^*_r(\mathcal{G}(X,d))$.
	
	Uniform Roe algebras encode a substantial amount of coarse-geometric information about the underlying space. For example, a uniformly locally finite metric space $(X,d)$ has property~A if and only if $C^*_u(X)$ is a nuclear $C^*$-algebra \cite{skandalis2002coarse}, and $(X,d)$ is amenable (as a metric space) if and only if $C^*_u(X)$ admits a tracial state \cite{ara2018amenability}. This naturally leads to the question of whether the uniform Roe algebra fully determines the coarse geometry of the underlying space. The authors of~\cite{baudier2022uniform} showed that the Morita equivalence class of the uniform Roe algebra is a complete invariant of coarse equivalence.
	
	\begin{theorem}[Baudier, Braga, Farah, Khukhro, Vignati, Willett, 2021]\label{Intro: theorem: rigidity}
		Let $(X,d)$ and $(Y,\partial)$ be uniformly locally finite metric spaces. The $C^*$-algebras $C^*_u(X)$ and $C^*_u(Y)$ are Morita equivalent if and only if $(X,d)$ and $(Y,\partial)$ are coarsely equivalent.
	\end{theorem}
	
	\noindent An analogous statement for isomorphisms of uniform Roe algebras and bijective coarse equivalences, known as ``the bijective rigidity problem'', was a long-standing conjecture in coarse geometry. It was recently resolved by Vignati in \cite{avignati}.
	
	\begin{theorem}[Vignati, 2026] \label{Intro: theorem: bij.rigidity}
		Let $(X,d)$ and $(Y,\partial)$ be uniformly locally finite metric spaces. The $C^*$-algebras $C^*_u(X)$ and $C^*_u(Y)$ are isomorphic if and only if $(X,d)$ and $(Y,\partial)$ are bijectively coarsely equivalent.
	\end{theorem}
	
	\vspace{0.5cm}
	
	The $K$-theory of the uniform Roe algebra is of particular importance both from a mathematical perspective and for applications in physics. In~\cite{vspakula2009uniform}, \v{S}pakula introduced the \emph{uniform coarse Baum--Connes conjecture}, which asks whether a certain assembly map from the uniform $K$-homology of a metric space to the $K$-theory of the associated uniform Roe algebra is an isomorphism. On the other hand, the Baum--Connes conjecture for ample groupoids yields another assembly map, from the topological $K$-theory of the coarse groupoid to the $K$-theory of the associated uniform Roe algebra, which was conjectured\footnote{
	Counterexamples to surjectivity were constructed in~\cite{higson2002counterexamples}.
	}
	to be an isomorphism. It was shown in~\cite{engel2019uniform} that, when the uniformly locally finite metric space arises from a finitely generated group, the two conjectures are equivalent. In~\cite{kubota2017controlled}, Kubota used uniform Roe algebras as models for controlled topological phases of bulk and edge quantum systems. In particular, the $K$-theory of the uniform Roe algebra plays a decisive role in the study of the so-called \emph{bulk index}.
	
	\subsection*{Main results}

	In this paper, we establish a connection between the bijective rigidity problem and the injectivity of the $0$th comparison map appearing in the HK~conjecture for coarse groupoids.

	We say that a collection $\mathcal{C}$ of metric spaces is \emph{admissible} (see Definition~\ref{definition: admissible class}) if it is stable under taking subspaces and doublings\footnote{		
	On the level of coarse groupoids, a doubling corresponds to taking the product with $\mathcal{R}_n$, where $\mathcal{R}_n$ denotes the discrete groupoid associated to the full equivalence relation on $n$ points. See Section~\ref{subsection: coarse geometry} for the coarse geometric definition.					
	}. For a property $\mathrm{P}$ of metric spaces, we say that an admissible collection $\mathcal{C}$ satisfies $\mathrm{P}$ if every element of $\mathcal{C}$ satisfies~$\mathrm{P}$. For example, since property~A passes to subspaces and doublings, the collection of metric spaces satisfying property~A forms an admissible collection. The collection of all uniformly locally finite metric spaces, clearly, forms an admissible collection. The following lemma constitutes the main result of this paper.
	
	\begin{llemma}\label{Intro: lemma: A}
		Let $\mathcal{C}$ be an admissible collection of uniformly locally finite metric spaces. The following statements are equivalent:
		\begin{enumerate}
			\item For every uniformly locally finite metric space $(X,d)$ in $\mathcal{C}$, the $0$th comparison map $\alpha_0 \colon H^{\mathrm{uf}}_0(X; \mathbb{Z}) \to K_0\bigl(C^*_u(X)\bigr)$ is injective;
			\item For any uniformly locally finite metric spaces $(X,d)$ and $(Y,\partial)$ in $\mathcal{C}$ such that $C^*_u(Y) \cong C^*_u(X)$, the coarse equivalence given by Theorem~\ref{Intro: theorem: rigidity} is close to a bijective coarse equivalence.
		\end{enumerate}
	\end{llemma}
	
	By Theorem \ref{Intro: theorem: bij.rigidity}, the collection of all uniformly locally finite metric spaces satisfies the second condition of Lemma \ref{Intro: lemma: A}; therefore, the first assertion holds unconditionally. Nevertheless, we formulate Lemma~\ref{Intro: lemma: A} as establishing a relationship between the two conjectural statements, as this viewpoint clarifies the techniques underlying our arguments. Moreover, we expect that the methods developed here can be adapted to broader classes of ample groupoids, potentially yielding new results on the injectivity of comparison maps and on rigidity phenomena for ample groupoids.
	
	\begin{ltheorem} \label{Intro: theorem: B}
		Let $(X,d)$ be a uniformly locally finite metric space, and $\mathcal{G}$ be its coarse groupoid. The comparison map $\alpha_0 \colon H_0(\mathcal{G}; \mathbb{Z}) \to K_0(C^*_r(\mathcal{G}))$ is injective. Moreover, if $(X,d)$ is coarsely connected, then $\alpha_0$ is split injective.
	\end{ltheorem}
	
	\noindent The above theorem partially extends the results of \cite{bonicke2023dynamic}, where the authors proved that $\alpha_0$ is split injective for uniformly locally finite metric spaces of small asymptotic dimension. To the best of our knowledge, no further results concerning the (split) injectivity of $\alpha_0$ for uniform Roe algebras are currently available. There exist many examples of uniformly locally finite metric spaces with infinite asymptotic dimension that are not coarsely connected, such as expander graphs. By contrast, any finitely generated group equipped with a word-length metric is coarsely connected. In this setting, for a finitely generated group $G$, the uniformly finite homology of $G$ is canonically isomorphic to the group homology of $G$ with coefficients in $\ell^{\infty}(G,\mathbb{R})$ (see \cite[Proposition A.10]{diana2017}).
	
	\begin{lcorollary} \label{Intro: corollary: C}
		Let $G$ be a finitely generated group equipped with the word-length metric. The canonical map
		\[
			\alpha_0^G \colon H_0\bigl(G; \ell^{\infty}(G,\mathbb{R})\bigr)
				\longrightarrow
			K_0\bigl(\ell^{\infty}(G) \rtimes_r G\bigr)
		\]
		is split-injective.
	\end{lcorollary}
	
	Recent results of~\cite{proietti2025cherncharactertorsionfreeample} show that the rational HK conjecture holds for all second-countable, locally compact, Hausdorff, ample groupoids that satisfy the rational Baum--Connes conjecture, and have torsion-free stabilisers. Coarse groupoids fall outside this framework: they are not second-countable and do not in general satisfy the Baum--Connes conjecture (see~\cite{higson2002counterexamples}). Nevertheless, by tensoring with $\mathbb{Q}$ in Theorem~\ref{Intro: theorem: B}, we immediately obtain unconditional split injectivity of the rational comparison map (see Remark~\ref{remark on rational injectivity}).
	
	\subsection*{Outline of the paper}

		In Section~\ref{section: preliminaries}, we introduce the necessary background from coarse geometry, uniform Roe algebras, and uniformly finite homology. We briefly recall the construction of the coarse groupoid from~\cite{skandalis2002coarse}, sketch the main ideas underlying the proof of Theorem~\ref{Intro: theorem: rigidity}, review the results on uniformly finite homology, and construct the $0$th comparison map.

		Section~\ref{section: uniform covering isometries} develops a framework of covering isometries for uniform Roe algebras, analogous to the construction for Roe algebras (see~\cite[Section~4.3]{willett2020higher}). Let $H$ be a separable Hilbert space. A \emph{uniform cover} of a uniformly finite-to-one coarse map $f \colon X \to Y$ is an isometry
		\[
			V_f \colon \ell^2(X,H) \longrightarrow \ell^2(Y,H)
		\]
		which acts fibrewise in a manner compatible with the action of $f$ on the underlying metric spaces, and is well-behaved with respect to finite-dimensional subspaces of the Hilbert space $H$ (see conditions~$(2)$ and~$(3)$ of Definition~\ref{definition: uniform cover}). This notion already appears implicitly in~\cite{brodzki2007property} in the special case of coarse equivalences. We show that any uniformly finite-to-one coarse map $f \colon (X,d) \to (Y, \partial)$ admits a uniform cover, and that the map $\operatorname{Ad}_{V_f}$ is a $*$-homomorphism between the stabilisations of the corresponding uniform Roe algebras. Moreover, the induced map in $K$-theory depends only on the closeness class of $f$. We further show that the $0$th comparison map defines a natural transformation from uniformly finite homology to the $K_0$-theory of uniform Roe algebras. The section concludes with a brief discussion of this framework from the perspective of étale groupoids.

	In Section~\ref{section: rigidity vs injectivity}, we prove Lemma~\ref{Intro: lemma: A}. The key step in the proof of the implication $(1)\Rightarrow(2)$ is to show that, given an isomorphism of uniform Roe algebras $\Phi \colon C^*_u(X) \longrightarrow C^*_u(Y)$, the induced map in $K$-theory coincides with the map arising from the coarse equivalence $f_\Phi \colon X \longrightarrow Y$ provided by Theorem~\ref{Intro: theorem: rigidity}. Since $K_0(\Phi)$ preserves the class of the unit, it follows that
	\[
		H_0^{\mathrm{uf}}(f_\Phi)([\mathbbm{1}_X]) = [\mathbbm{1}_Y]
	\]
	modulo $\ker(\alpha_0^Y)$. As the comparison maps are assumed to be injective, the result of \cite[Theorem A]{whyte99amenablity} (see Theorem \ref{theorem: Block-Weinberger-Whyte}) completes the proof of the implication $(1)\Rightarrow(2)$. For the converse implication, to each positive function $h \in \ell^{\infty}(X, \mathbb{Z})$ we associate a subspace $X(h)$ of a doubling of $X$. We show that if $h_1, h_2$ are positive functions such that
	\[
		\alpha_0^X([h_1]) = \alpha_0^X([h_2]),
	\]
	then the uniform Roe algebras of $X(h_1)$ and $X(h_2)$ are isomorphic. By the already established implication, there exists a bijective coarse equivalence $f \colon X(h_1) \to X(h_2)$. The graph of $f$ provides a cycle $\delta \in C_1^{\mathrm{uf}}(X;\mathbb{Z})$ such that
	\[
		\partial_1(\delta) = [h_1] - [h_2].
	\]
	This completes the proof of Lemma~\ref{Intro: lemma: A}. In the final section, we prove Theorem~\ref{Intro: theorem: B} and discuss further consequences of Lemma~\ref{Intro: lemma: A}.
	
	\vspace{0.4cm}
	
	\begin{center}
		\textbf{Acknowledgements}
	\end{center}
	
	The author is grateful to his advisors, Alessandro Vignati and Romain Tessera, for their advice on the presentation of this work. He also thanks Ján~\v{S}pakula and the University of Southampton for their hospitality during his research visit. The author is indebted to Rufus~Willett for valuable comments on Remark~\ref{remark: Baum-Connes}. This work was partially supported by the travel grant ECOST Actions CaLISTA (E-COST-GRANT-CA21109-15f1764d) and by the ANR grant ROAR (ANR-25-CE40-5029)
.	
	\vspace{1 cm}
	
	\setcounter{section}{0}
	
	\section{Preliminaries} \label{section: preliminaries}
	
	In this section, we introduce the basic notions from coarse geometry, uniform Roe algebras, and uniformly finite homology. For a more detailed exposition, we refer the reader to~\cite{roe2003lectures}. Let $X$, $Y$, and $Z$ be sets. For a subset $E \subset X \times Y$ we define its \emph{adjoint} as
	\[
		E^{-1} := \{(y,x) \mid (x,y) \in E\}.
	\]
	For subsets $E \subset X \times Y$ and $F \subset Y \times Z$ we define their \emph{composition} as
	\[
		E \circ F := \{(x,z) \mid \text{for some } y \in Y \text{ such that } (x,y) \in E, \, (y,z) \in F\}.
	\]
	The subset $\Delta_X =\{(x,x) \mid x \in X\} $ of $X \times X$ is called the \emph{diagonal of $X$}. Given a map $f \colon X \to Y$, we regard its graph $\operatorname{Graph}(f)$ as a subset of $Y \times X$. By a \emph{partially defined map} $t \colon X \dashrightarrow Y$ we mean a triple 
	$$
		\bigl( \operatorname{dom}(t), t, \operatorname{ran}(t)\bigr),
	$$
	where $\operatorname{dom}(t)$ and $\operatorname{ran}(t)$ are subsets of $X$ and $Y$, respectively, and $t$ is map from $\operatorname{dom}(t)$ to $\operatorname{ran}(t)$.

	\subsection{Coarse geometry} \label{subsection: coarse geometry}
	
	Coarse geometry studies geometric objects from the perspective of their large-scale structure. One of its principal objects of interest is the class of uniformly locally finite metric spaces.

	\begin{definition}\label{definition: ulf space}
		A metric space $(X,d)$ is said to be \emph{uniformly locally finite} if for every $R \ge 0$ there exists $N \in \mathbb{N}$ such that
		\[
			|B_R(x)| \le N
		\]
		for all $x \in X$, where $B_R(x)$ denotes the ball of radious $R$ centered at $x$.
	\end{definition}

	In the literature, uniformly locally finite metric spaces are also referred to as \emph{bounded geometry} spaces. For example, any connected uniformly locally finite graph is a uniformly locally finite metric space, when equipped with the shortest path metric; in particular, the Cayley graph of a finitely generated group provides a canonical example. Discretisations of bounded geometry Riemannian manifolds yield further examples. To define appropriate morphisms between uniformly locally finite metric spaces, the notion of asymptotic sets plays a central role.
	
	\begin{definition}
		Let $(X,d)$ and $(Y,\partial)$ be uniformly locally finite metric spaces. Two subsets $S_1,S_2 \subset Y \times X$ are said to be \emph{asymptotic}, denoted $S_1 \asymp S_2$, if there exists $R \ge 0$ such that
		\[
			S_2 \subset \bigcup_{(y,x)\in S_1} \bigl(B_R(y) \times B_R(x)\bigr)
			\quad \text{and} \quad
			S_1 \subset \bigcup_{(y,x)\in S_2} \bigl(B_R(y) \times B_R(x)\bigr).
		\]
	\end{definition}
	
	Informally, two subsets of $Y \times X$ are asymptotic if each is contained in a bounded neighbourhood of the other. A subset of $X \times X$ that is asymptotic to the diagonal $\Delta_X$ is called an \emph{entourage}. It is straightforward to verify that asymptoticity defines an equivalence relation. Moreover, for metric spaces $(X,d)$ and $(Y,\partial)$ and maps $f,g \colon X \to Y$, the graphs $\operatorname{Graph}(f)$ and $\operatorname{Graph}(g)$ are asymptotic if and only if the maps are uniformly close, that is,
	\[
		\sup_{x \in X} \partial\bigl(f(x),g(x)\bigr) < \infty.
	\]
	In this case, we say that $f$ is \emph{close} to $g$. The notion of closeness is compatible with composition: the closeness class of a composition depends only on the closeness classes of the constituent maps.

	\begin{definition}
		Let $(X,d)$ and $(Y,\partial)$ be uniformly locally finite metric spaces. A map $f \colon X \to Y$ is said to be \emph{coarse} if for every $R \ge 0$ there exists $S \ge 0$ such that
		\[
			f\bigl(B_R(x)\bigr) \subset B_S\bigl(f(x)\bigr)
			\quad \text{for all } x \in X.
		\]
		A coarse map $f \colon X \to Y$ is called a \emph{coarse equivalence} if there exists a coarse map $g \colon Y \to X$ such that the compositions $f \circ g$ and $g \circ f$ are close to $\operatorname{id}_Y$ and $\operatorname{id}_X$, respectively. In this case, $g$ is called a \emph{coarse inverse} of $f$. A coarse map $f \colon X \to Y$ is said to be \emph{uniformly finite-to-one} if
		\[
			\sup_{y \in Y} \bigl|f^{-1}(\{y\})\bigr| < \infty.
		\]
	\end{definition}
	We say that two uniformly locally finite metric spaces $(X,d)$ and $(Y,\partial)$ are \emph{(bijectively) coarsely equivalent} if there exists a (bijective) coarse equivalence
	\[
		f \colon (X,d) \longrightarrow (Y,\partial).
	\]
	It is straightforward to verify that every coarse equivalence between uniformly locally finite metric spaces is uniformly finite-to-one. We define the category $\mathbf{Coarse}$ to have uniformly locally finite metric spaces as objects and closeness classes of uniformly finite-to-one coarse maps as morphisms. Isomorphisms in this category are precisely the coarse equivalences.
	
	Given a uniformly locally finite metric space $(X,d)$ and a subset $A \subset X$, the restriction of the metric $d$ to $A$ turns $(A,d)$ into a uniformly locally finite metric space. Another construction that will play an important role is that of a \emph{doubling}. For a uniformly locally finite metric space $(X,d)$, its \emph{$n$th doubling} is the uniformly locally finite metric space $(X^{(n)},d_n)$ defined by
	\[
		X^{(n)} := X \times \{1,\ldots,n\}, 
			\qquad 
		d_n\bigl((x,i),(y,j)\bigr) := d(x,y) + |i-j|,
	\]
	for $x,y \in X$ and $1 \le i,j \le n$. It is straightforward to check that, for each $1 \le k \le n$, the inclusion
	\[
		i_k \colon X \longrightarrow X^{(n)}, \qquad x \longmapsto (x,k),
	\]
	is a coarse equivalence. The projection map $p \colon X \times \{1, \ldots, n\} \to X$ constitutes its coarse inverse.

	For a uniformly locally finite metric space $(X,d)$, the authors of \cite{skandalis2002coarse} constructed a topological groupoïd that captures the coarse geometry of $(X,d)$. In this paper, we do not appeal to the theory of topological groupoids as a primary tool in our proofs. Nevertheless, we believe that it is instructive to highlight the connection between coarse geometry and topological groupoids. For a detailed introduction to topological groupoids, we refer the reader to \cite{sims2017etale}. A \emph{groupoid} is a small category in which every morphism is invertible. We shall identify a groupoid with its set of morphisms, denoted by $\mathcal{G}$. The set of objects, or \emph{units}, is denoted by $\mathcal{G}^{(0)}$ and is identified with the subset of identity morphisms in $\mathcal{G}$. The groupoid $\mathcal{G}$ is equipped with the range and source maps
	\[
		r,s \colon \mathcal{G} \longrightarrow \mathcal{G}^{(0)},
	\]
	which assign to each arrow its range and source, respectively, together with a partially defined multiplication
	\[
		\mathcal{G} {_s}\times_r \mathcal{G}
		:= \{(\gamma_1,\gamma_2) \in \mathcal{G}\times\mathcal{G} \mid s(\gamma_1)=r(\gamma_2)\}
		\longrightarrow \mathcal{G}, 
		\qquad (\gamma_1,\gamma_2) \longmapsto \gamma_1\gamma_2,
	\]
	and an inversion map $\mathcal{G} \to \mathcal{G}$, $\gamma \mapsto \gamma^{-1}$. These maps satisfy the usual axioms that make $\mathcal{G}$ into a category. We shall be concerned with \emph{topological groupoids}. A \emph{locally compact, Hausdorff groupoid} is a groupoid $\mathcal{G}$ endowed with a locally compact, Hausdorff topology such that the range and source maps, the multiplication, and the inversion are continuous. Here $\mathcal{G} {_s}\times_r \mathcal{G}$ is equipped with the subspace topology inherited from the product topology on $\mathcal{G}\times\mathcal{G}$. A locally compact, Hausdorff groupoid $\mathcal{G}$ is said to be \emph{étale}\footnote{	
		Many authors do not require étale groupoids to be Hausdorff. Throughout this paper, we restrict attention to the Hausdorff case.		
		} if its range map (equivalently, its source map) is a local homeomorphism. A groupoid $\mathcal{G}$ is said to be \emph{principal} if its isotropy bundle
	\[
		\operatorname{Iso}(\mathcal{G}) := \{\gamma \in \mathcal{G} \mid s(\gamma)=r(\gamma)\}
	\]
	coincides with $\mathcal{G}^{(0)}$, i.e. the only arrows whose source and range coincide are the identity arrows. An étale groupoid is said to be \emph{ample} if its unit space $\mathcal{G}^{(0)}$ is totally disconnected.
	
	In \cite{skandalis2002coarse}, Skandalis, Tu, and Yu associated to every uniformly locally finite metric space $(X,d)$ a topological groupoid $\mathcal{G}(X,d)$, defined as follows. For a countable discrete topological space $A$, let $\beta A$ denote its Stone--Čech compactification. Given an entourage $E \subset X \times X$, denote by $\overline{E}$ its closure in $\beta(X \times X)$. One then defines a topological space
	\[
		\mathcal{G}(X,d)
			= \bigcup_{\substack{E \text{ is an} \\ \text{entourage}}} \overline{E}
			\;\subset\; \beta(X \times X),
	\]
	equipped with the weak topology. The pair groupoid $X \times X$ embeds as a dense open subset of $\mathcal{G}(X,d)$. It is shown in \cite{skandalis2002coarse} that the groupoid operations on $X \times X$ extend continuously to $\mathcal{G}(X,d)$, endowing it with the structure of a Hausdorff, principal, ample groupoid. We shall refer to $\mathcal{G}(X,d)$ as the \emph{coarse groupoid} associated to $(X,d)$. Note that $\mathcal{G}(X,d)$ is not second countable unless the metric space $(X,d)$ is finite.
	
	\subsection{Uniform Roe algebras and rigidity}
	
	Uniform Roe algebras are $C^*$-algebras associated with uniformly locally finite metric spaces. Let $H$ be a separable Hilbert space and let $X$ be a set. Denote by $\ell^2(X, H)$ the Hilbert space of all square-summable functions from $X$ to $H$. There is a canonical identification
	\[
		\ell^2(X,H) \cong \ell^2(X) \otimes H.
	\]
	For a subset $A \subseteq X$, let $\mathbbm{1}_A$ denote the orthogonal projection onto $\ell^2(A,H)$. If $A = \{x\}$ is a singleton, we simply write $\mathbbm{1}_x$. The main ingredient in the definition of uniform Roe algebras is the notion of support.
	
	\begin{definition}
		Let $(X,d)$ and $(Y,\partial)$ be uniformly locally finite metric spaces, and let $H$ be a separable Hilbert space. The \emph{support} of a bounded operator $T \colon \ell^2(X, H) \to \ell^2(Y, H)$ is the set
		\[
			\operatorname{supp}(T) := 
			\{(y,x) \in Y \times X \mid \|\mathbbm{1}_y T \mathbbm{1}_x\| \neq 0\}.
		\]
		We say that a bounded operator $T \colon \ell^2(X, H) \to \ell^2(Y, H)$ is \emph{controlled} if its support is asymptotic to a coarse map. An operator $T \in \mathcal{B}(\ell^2(X, H))$ is said to have \emph{controlled propagation} if its support is asymptotic to $\operatorname{id}_X$ (equivalently, if its support is an entourage). It is said to be \emph{approximable} if it is a norm limit of controlled propagation operators.
	\end{definition}
	
	It is straightforward to verify that the support behaves well with respect to compositions, adjoints, and sums. More precisely, let $(X,d_X)$, $(Y,d_Y)$, and $(Z,d_Z)$ be uniformly locally finite metric spaces, and let
	\[
		T_1, T_2 \colon \ell^2(X,H) \to \ell^2(Y,H), 
		\qquad 
		S \colon \ell^2(Y,H) \to \ell^2(Z,H)
	\]
	be bounded operators. Then the supports of sums, products, and adjoints satisfy the following relations:
	\begin{equation*}
		\begin{split}
			& \operatorname{supp}(T_1 + T_2)
			 \subset \operatorname{supp}(T_1) \cup \operatorname{supp}(T_2), \\
			& \operatorname{supp}(S T_1)
			 \subset \operatorname{supp}(S) \circ \operatorname{supp}(T_1), \\
			&\operatorname{supp}(T_1^*)
			 = \operatorname{supp}(T_1)^{-1}.
		\end{split}
	\end{equation*}
	In particular, the controlled propagation operators on $\ell^2(X,H)$ form a $*$-subalgebra of $\mathcal{B}(\ell^2(X,H))$. Controlled propagation operators in $\mathcal{B}(\ell^2(X, H))$ can be naturally quantified by the extent to which their support is spread. For a bounded operator $T \in \mathcal{B}(\ell^2(X,H))$, we define its \emph{propagation} by
	\[
		\operatorname{prop}(T)
		= \sup \{ d(x,x') \mid (x,x') \in \operatorname{Supp}(T) \}.
	\]
	Note that an operator $T$ has \emph{controlled propagation} if and only if $\operatorname{prop}(T) < \infty$.

	\begin{definition}
		Let $(X,d)$ be a uniformly locally finite metric space. The \emph{uniform Roe algebra} $C^*_u(X)$ of $(X,d)$ is the norm closure in $\mathcal{B}(\ell^2(X))$ of the $*$-algebra of bounded operators of controlled propagation.
	\end{definition}
	
	Uniform Roe algebras provide a natural bridge between coarse geometry and operator algebras, as their structural properties encode many coarse invariants. Alternatively, uniform Roe algebras can be defined as the reduced groupoid $C^*$-algebras of the coarse groupoids associated to uniformly locally finite metric spaces (see \cite[Proposition~10.29]{roe2003lectures}). Note that the uniform Roe algebra of a doubling of $X$ is canonically isomorphic to the matrix algebra over $C^*_u(X)$, that is, for every $n \ge 1$ one has:
	\[
		C^*_u(X^{(n)}) \cong C^*_u(X) \otimes M_n(\mathbb{C}).
	\]
	We shall also consider stabilisations of uniform Roe algebras. Let $H$ be a separable, infinite-dimensional Hilbert space. The $C^*$-algebra $C^*_u(X) \otimes \mathbb{K}$ admits a natural faithful representation on $\ell^2(X, H)$. For each $x \in X$, consider an isometry 
		\[
			v_x \colon H \longrightarrow \ell^2(X,H), \qquad w \longmapsto \delta_x \otimes w.
		\]
	For a finite-dimensional subspace $W \subset H$ and $R \ge 0$, consider the following set of bounded operators on $\ell^2(X,H)$:
		\[
			\mathbb{C}[X,R,W] := \bigl\{ T \;\big|\; \operatorname{prop}(T) \le R \text{ and } v_x^* T v_y \in \mathcal{B}(W) \text{ for all } x,y \in X \bigr\}.
		\]
	The following statement summarises the discussion in \cite[Section 4]{baudier2022uniform}, providing a convenient dense $*$-subalgebra of $C^*_u(X) \otimes \mathbb{K}$.

	\begin{proposition}[{\cite[Section 4]{baudier2022uniform}}] \label{proposition: dense subset of the stabilisation}
		Let $(X,d)$ be a uniformly locally finite metric space, and let $H$ be a separable Hilbert space.
		Then the union of the sets $\mathbb{C}[X, R, W]$ over all finite-dimensional subspaces $W \subset H$ and all $R \ge 0$ is dense in $C^*_u(X) \otimes \mathbb{K}$.
	\end{proposition}
	
	In particular, any approximable operator $T \in \mathcal{B}(\ell^2(X, H))$ for which there exists a finite-dimensional subspace $W \subset H$ such that $v_x^* T v_y$ belongs to $\mathcal{B}(W)$ for all $x,y \in X$ belongs to the stabilisation $C^*_u(X) \otimes \mathbb{K}$ of the uniform Roe algebra.
	
	The rigidity problem for uniform Roe algebras asks whether Morita equivalences (respectively, isomorphisms) of uniform Roe algebras induce coarse equivalences (respectively, bijective coarse equivalences) between the underlying uniformly locally finite metric spaces. In \cite{baudier2022uniform}, the authors provided an affirmative answer to the rigidity problem in the case of Morita equivalences.
	
	\begin{theorem}[{\cite[Theorem 1.4]{baudier2022uniform}}] \label{theorem: rigidity theorem}
		Let $(X,d)$ and $(Y,\partial)$ be uniformly locally finite metric spaces. Then the $C^*$-algebras $C^*_u(X)$ and $C^*_u(Y)$ are Morita equivalent if and only if $(X,d)$ and $(Y,\partial)$ are coarsely equivalent.
	\end{theorem}
	
	The authors of \cite{martinez2025c} established an analogue of Theorem~\ref{theorem: rigidity theorem} for a wide class of Roe-type algebras. Their methods are also applicable to stabilisations of uniform Roe algebras. In what follows, we recall the strategy of the proof of Theorem~\ref{theorem: rigidity theorem} from \cite{martinez2025c} in the special case of an isomorphism $\Phi \colon C^*_u(X) \to C^*_u(Y)$.
	\begin{enumerate}
		\item One proves that the isomorphism $\Phi$ is \emph{spatially implemented}, in the sense that there exists a unitary operator $U \colon \ell^2(X) \to \ell^2(Y)$ such that $\Phi = \operatorname{Ad}_U$ (see \cite[Lemma 3.1]{vspakula2013rigidity});
		\item One shows that the unitaries $U$ and $U^*$ implementing $\Phi$ are \emph{coarse-like}\footnote{ 
		In \cite{martinez2025c, martinez2025rigidityframeworkroelikealgebras}, this notion is also referred to as \emph{weakly approximately controlled}.	
		} (see \cite[Theorem 3.5]{braga2023gelfand}). Recall from~\cite[Definition 3.1]{braga2023gelfand} that an operator 
			\[
				Q \colon \ell^2(X) \longrightarrow \ell^2(Y)
			\]
			is said to be \emph{coarse-like} if, for every $R \ge 0$ and every $\varepsilon > 0$, there exists $S \ge 0$ such that for every contraction $T \in \mathcal{B}(\ell^2(X))$ of propagation less than $R$ the operator $\operatorname{Ad}_Q(T)$ is $\varepsilon$-close (in norm) to an operator on $\ell^2(Y)$ of propagation at most $S$.
		\item One shows that, for every $\delta \in (0,1)$, there exist constants $R,S \ge 0$ such that the subset
			\[
				\quad \qquad f_{\delta,R,S}^U = \bigcup 
					\left\{
						B \times A
							\;\middle|\;
						\operatorname{diam}(B) \le R,\ 
						\operatorname{diam}(A) \le S,\ 
						\bigl\| \mathbbm{1}_B\, U\, \mathbbm{1}_A \bigr\| > \delta
					\right\}
			\]
			of $Y \times X$ is asymptotic to the graph of a coarse equivalence $f_{\Phi}$ between $(X,d)$ and $(Y, \partial)$.
	\end{enumerate}

	The resulting coarse equivalence $f_{\Phi} \colon (X,d) \to (Y,\partial)$ is uniquely determined up to closeness. We shall refer to $f_{\Phi}$ as the \emph{coarse equivalence induced by $\Phi$}. The following theorem provides a key tool for relating the coarse equivalence $f_{\Phi}$ back to the isomorphism $\Phi$.

	\begin{theorem}[{\cite[Theorem 4.5]{martinez2025c}}] \label{theorem: norm limit of unitaries supported on f_Phi}
		Let $(X,d)$ and $(Y,\partial)$ be uniformly locally finite metric spaces, and let $\Phi \colon C^*_u(X) \to C^*_u(Y)$ be an isomorphism. Let $U$ be a unitary implementing $\Phi$, and let $f_{\Phi}$ be a coarse equivalence induced by $\Phi$. Then, for an infinite-dimensional Hilbert space $H$, the unitary $U \otimes \operatorname{id}_H$ is a norm limit of unitaries supported on subsets of $Y \times X$ that are asymptotic to the graph of $f_{\Phi}$.
	\end{theorem}

	\subsection{Uniformly finite homology} \label{subsection: UF homology}
	
	Uniformly finite homology was introduced by Block and Weinberger in~\cite{block1992aperiodic} to study tilings of non-amenable polyhedra and characteristic numbers of manifolds whose universal covers admit metrics of positive scalar curvature. It was subsequently employed by Whyte in~\cite{whyte99amenablity} to prove the geometric von Neumann conjecture.
	
	\begin{definition} \label{definition: Z-valued controlled propagation function}
		Let $(X,d)$ be a uniformly locally finite metric space, and let $n \ge 1$ be an integer. A function $f \colon X^{n} \to \mathbb{Z}$ is said to have \emph{controlled propagation} if there exists $R \ge 0$ such that
	\[
		f(x_1,\ldots,x_n) = 0 \quad \text{whenever} \quad \max_{i,j} d(x_i,x_j) > R.
	\]
	\end{definition}
	
	Let $(X,d)$ be a uniformly locally finite metric space. For each $n \ge 0$, define the abelian group of \emph{uniformly finite $n$-chains} $C_n^{\mathrm{uf}}(X;\mathbb{Z})$ to consist of all bounded, controlled-propagation functions $f \colon X^{n+1} \to \mathbb{Z}$. For $n \ge 0$ and $f \in C_{n+1}^{\mathrm{uf}}(X;\mathbb{Z})$, define the boundary operator
	\[
		\partial_{n+1}(f)(x_0,\ldots,x_n) =
			\sum_{i=0}^{n+1} (-1)^i
				\sum_{y \in X}
					f(x_0,\ldots,x_{i-1},y,x_i,\ldots,x_n),
	\]
	where, in the inner sum, the variable $y$ is inserted in the $i$-th coordinate. The map $\partial_{n+1}$ is a well-defined group homomorphism
	\[
		\partial_{n+1} \colon C_{n+1}^{\mathrm{uf}}(X;\mathbb{Z}) \longrightarrow C_n^{\mathrm{uf}}(X;\mathbb{Z}),
	\]
	and the identities $\partial_n \circ \partial_{n+1} = 0$ hold for all $n \ge 1$. Consequently, the collection $\bigl(C_*^{\mathrm{uf}}(X;\mathbb{Z}), \partial_*\bigr)$ forms a chain complex of abelian groups, which we call the \emph{uniformly finite chain complex} of $(X,d)$.
	
	\begin{definition}\label{definition: uf-homology}
		Let $(X,d)$ be a uniformly locally finite metric space. The \emph{$n$th uniformly finite homology group} of $(X,d)$ with integer\footnote{ 
		One can similarly define the uniformly finite chain complex and uniformly finite homology with coefficients in $\mathbb{R}$. 
		} coefficients, denoted $H_n^{\mathrm{uf}}(X;\mathbb{Z})$, is defined to be the $n$th homology group of the uniformly finite chain complex of $(X,d)$.
	\end{definition}

	Let $(X,d)$ and $(Y,\partial)$ be uniformly locally finite metric spaces, and let $f \colon X \to Y$ be a uniformly finite-to-one coarse map. Define a homomorphism
	\[
		f_* \colon C_n^{\mathrm{uf}}(X;\mathbb{Z}) \longrightarrow C_n^{\mathrm{uf}}(Y;\mathbb{Z}), \quad f_*(g)(\bar{y})
		=
		\sum_{f(\bar{x})=\bar{y}} g(\bar{x}), \, \text{ for } \, \bar{y} \in Y^{n+1}.
	\]
	It is straightforward to check that $f_*$ is a well-defined group homomorphism and that it commutes with the boundary operators. Consequently, it induces a homomorphism in homology,
	\[
		H_n^{\mathrm{uf}}(f) \colon H_n^{\mathrm{uf}}(X;\mathbb{Z}) \longrightarrow H_n^{\mathrm{uf}}(Y;\mathbb{Z}).
	\]
	Moreover, if two uniformly finite-to-one coarse maps $f,g \colon (X,d) \to (Y,\partial)$ are close, then the induced maps in uniformly finite homology coincide. In particular, for each $n \ge 0$, uniformly finite homology defines a functor
	\[
		H_n^{\mathrm{uf}} \colon \mathbf{Coarse} \longrightarrow \mathbf{Ab}
	\]
	from the coarse category to the category of abelian groups.

	In this paper, we shall be primarily concerned with the case $n=0$. Recall from~\cite[Definition 10.21]{roe2003lectures} that a \emph{partial translation} on a metric space $(X,d)$ is a partially defined bijection
	\[
		t \colon X \dashrightarrow X \qquad \text{that satisfies} \quad \sup_{x \in \operatorname{dom}(t)} d(x,t(x)) < \infty.
	\]
	Note that the above condition holds if and only if the indicator function of $\mathbbm{1}_{\operatorname{Graph}(t)}$ belongs to $C_1^{\text{uf}}(X; \mathbb{Z})$. From the coarse-geometric perspective, a partial translation exhibits its domain and range as identical subspaces of $(X,d)$. Given a partial translation $t \colon X \dashrightarrow X$, the indicator function of its graph,
	\[
		\mathbbm{1}_{\operatorname{Graph}(t)} \colon X \times X \longrightarrow \mathbb{Z},
	\]
	is a controlled propagation function in the sense of Definition~\ref{definition: Z-valued controlled propagation function}. A direct computation shows that
	\[
		\partial_1\!\left(\mathbbm{1}_{\operatorname{Graph}(t)}\right)
		=
		\mathbbm{1}_{\operatorname{dom}(t)} - \mathbbm{1}_{\operatorname{ran}(t)}.
	\]
	In particular, if two subsets of $X$ are related by a partial translation, then they define the same class in the uniformly finite homology group $H_0^{\mathrm{uf}}(X;\mathbb{Z})$.
	
	\begin{remark} \label{remark: 1-cycles are generated by part. trans}
		Let $(X,d)$ be a uniformly locally finite metric space. The group $C_1^{\mathrm{uf}}(X; \mathbb{Z})$ is generated by indicator functions of entourages $E \subset X \times X$. The argument of~\cite[Lemma 4.10]{roe2003lectures} shows that for any such entourage $E$ there exist partial translations $t_1, \ldots, t_n \colon X \dashrightarrow X$ whose graphs are pairwise disjoint subsets of $E$ and satisfy
	\[
		E = \bigsqcup_{i=1}^n \operatorname{Graph}(t_i).
	\]
	In particular, the indicator function of $E$ decomposes as a finite sum of indicator functions of the graphs $\operatorname{Graph}(t_i)$. Consequently, the group $C_1^{\mathrm{uf}}(X; \mathbb{Z})$ is generated by indicator functions of graphs of partial translations of $(X,d)$.
	\end{remark}
	
	The homology class of $\mathbbm{1}_X$ in $H_0^{\mathrm{uf}}(X; \mathbb{Z})$ is referred to as the \emph{fundamental class} of $(X,d)$. It plays a central role in determining which coarse equivalences are close to bijective ones. This is made precise by the following theorem of Whyte.
	
	\begin{theorem}[{\cite[Theorem A]{whyte99amenablity}}] \label{theorem: Block-Weinberger-Whyte}
		Let $(X,d)$, $(Y,\partial)$ be uniformly locally finite metric spaces, and let $f \colon X \to Y$ be a coarse equivalence. The following are equivalent:
		\begin{enumerate}
			\item There exists a bijective coarse equivalence $\tilde f \colon X \to Y$ that is close to $f$;
			\item $H^{\mathrm{uf}}_0(f)\big([\mathbbm{1}_X]\big) = [\mathbbm{1}_Y]$.
		\end{enumerate}
	\end{theorem}
	
	It is straightforward to verify that for a finite metric space $(X,d)$, one has $H_0^{\mathrm{uf}}(X;\mathbb{Z}) \cong \mathbb{Z}$. In contrast, the $0$th uniformly finite homology group is often highly non-trivial for infinite spaces. For example, $H_0^{\mathrm{uf}}(\mathbb{Z};\mathbb{Z})$ is an infinite-dimensional real vector space.
	
	\begin{definition} \label{definition: coarsely connected}
		A uniformly locally finite metric space $(X,d)$ is said to be \emph{coarsely connected} if there exists $R \ge 0$ such that for any $x,y \in X$ there are points $x_0, x_1, \ldots, x_{n+1} \in X$ satisfying
		$$
			x = x_0, \qquad y = x_{n+1}, \qquad d(x_i, x_{i+1}) \le R, \, \text{ for all } 0 \le i \le n.
		$$
	\end{definition}
	
	The following theorem shows that, for infinite coarsely connected metric spaces, the $0$th uniformly finite homology group is either trivial or a real vector space.
	
	\begin{theorem}[{\cite[Corollary 3.1 and Corollary 3.2]{whyte99amenablity}}] \label{theorem: injectivity/isomorphism of extension of scalars}
		Let $(X,d)$ be an infinite uniformly locally finite metric space. Then the extension of scalars
		\[
			i \colon H_0^{\mathrm{uf}}(X;\mathbb{Z}) \longrightarrow H_0^{\mathrm{uf}}(X;\mathbb{R})
		\]
		is an injective group homomorphism. Moreover, if $(X,d)$ is coarsely connected, then $i$ is an isomorphism.
	\end{theorem}
	
	In particular, if $(X,d)$ is an amenable, coarsely connected, uniformly locally finite metric space, then $H_0^{\mathrm{uf}}(X;\mathbb{Z})$ is a non-zero real vector space; see~\cite[Theorem~D]{whyte99amenablity}. It also follows from the above result that the $0$th uniformly finite homology group is always torsion-free.
	
	Uniformly finite homology provides information about the structure of the $K_0$-group of the uniform Roe algebra. The $K_0$-group of $\ell^{\infty}(X)$ identifies with $\ell^{\infty}(X, \mathbb{Z})$, which, by definition, coincides with the group of uniformly finite $0$-chains $C_0^{\mathrm{uf}}(X; \mathbb{Z})$. The canonical inclusion $i \colon \ell^{\infty}(X) \hookrightarrow C^*_u(X)$ induces a group homomorphism 
	\[
		K_0(i) \colon \ell^{\infty}(X, \mathbb{Z}) \longrightarrow K_0(C^*_u(X)).
	\] 
	It is straightforward to verify that the composition $K_0(i) \circ \partial_1$ vanishes: By Remark~\ref{remark: 1-cycles are generated by part. trans}, the indicator functions of partial translations generate $C_1^{\mathrm{uf}}(X; \mathbb{Z})$; therefore, it suffices to show that for any partial translation $t \colon X \dashrightarrow X$ one has
	\begin{equation} \label{equation: K_0(i) circ partial_1 = 0}
		K_0(i)(\mathbbm{1}_{\operatorname{dom}(t)}) - K_0(i)(\mathbbm{1}_{\operatorname{ran}(t)}) = K_0(i) \circ \partial_1(\mathbbm{1}_{\operatorname{Graph}(t)}) = 0.
	\end{equation}
	It is enough to find a partial isometry in the uniform Roe algebra that intertwines the two projections. Let $v \in \mathcal{B}(\ell^2(X))$ be the partial isometry defined by
	\[
		v(\delta_x) =
			\begin{cases}
				\delta_{t(x)} & \text{if } x \in \operatorname{dom}(t), \\
				0 & \text{otherwise.}
			\end{cases}
	\]
	Then $vv^* = \mathbbm{1}_{\operatorname{ran}(t)}$ and $v^*v = \mathbbm{1}_{\operatorname{dom}(t)}$. Since $t$ is a partial translation, the partial isometry $v$ has finite propagation and therefore belongs to the uniform Roe algebra of $(X,d)$. Consequently, \eqref{equation: K_0(i) circ partial_1 = 0} holds.
	
	\begin{definition} \label{definition: comparison map}
		Let $(X,d)$ be a uniformly locally finite metric space. The canonical map $K_0(i) \colon \ell^{\infty}(X, \mathbb{Z}) \to K_0(C^*_u(X))$ descends to a group homomorphism
		\[
			\alpha_0 \colon H_0^{\mathrm{uf}}(X; \mathbb{Z}) \longrightarrow K_0(C^*_u(X)),
		\]
		which will be referred to as the \emph{$0$th comparison map}.
	\end{definition}
	
	The map $\alpha_0$ coincides with the $0$th comparison map appearing in Conjecture~\ref{conjecture: HK} for coarse groupoids. It remains an open problem whether this map is split injective in general.
	
	\section{Uniform covering isometries} \label{section: uniform covering isometries}
	
	Let $H$ be a separable Hilbert space. It is well known (for instance, see \cite[Section 4.3]{willett2020higher}) that, for a proper coarse map between two uniformly locally finite metric spaces $f \colon (X,d) \to (Y,\partial)$, there exists an isometry $S \colon \ell^2(X, H) \to \ell^2(Y, H)$, called a \emph{covering isometry of $f$}, such that the support of $S$ is asymptotic to the graph of $f$ and $\operatorname{Ad}_S$ restricts to a $*$-homomorphism between the associated Roe algebras. Moreover, the induced map in $K$-theory depends only on $f$. However, covering isometries do not, in general, induce $*$-homomorphisms between uniform Roe algebras or their stabilisations. In this section, we introduce a modification of the above notion adapted to the setting of stabilised uniform Roe algebras.
	
	\begin{definition} \label{definition: uniform cover}
		Let $(X,d)$ and $(Y,\partial)$ be uniformly locally finite metric spaces, and let $H$ be a separable Hilbert space. Given a coarse map $f \colon (X,d) \to (Y,\partial)$, an isometry $S \colon \ell^2(X,H) \to \ell^2(Y,H)$ is called a \emph{uniform cover} if the following conditions are satisfied:
		\begin{enumerate}
			\item There exists $R \ge 0$ such that, for any $x \in X$ and $y \in Y$ with $\|\mathbbm{1}_{y} S \mathbbm{1}_{x}\| \neq 0$, one has $d(y,f(x)) \le R$; in other words, $S$ covers $f$.
			\item For any finite-dimensional subspace $V \subset H$, there exists a finite-dimensional subspace $W \subset H$ such that $S$ restricts to a map $\ell^2(X,V) \to \ell^2(Y,W)$.
			\item For any finite-dimensional subspace $W \subset H$, there exists a finite-dimensional subspace $V \subset H$ such that $S^*$ restricts to a map $\ell^2(Y,W) \to \ell^2(X,V)$.
		\end{enumerate}
	\end{definition}

	The first condition in the above definition asserts that $S$ is a covering isometry of $f$. The motivation for the last two conditions is that we require the isometries to induce maps between the filtrations of $C^*_u(X) \otimes \mathbb{K}$ introduced in Proposition~\ref{proposition: dense subset of the stabilisation}. More precisely, for every $R \ge 0$ and a finite-dimensional subspace $V \subset H$, there should exist $S \ge 0$ and a finite dimensional subspace $W \subset H$ such that $\operatorname{Ad}_S$ restricts to a map
	$$
		\operatorname{Ad}_S \colon \mathbb{C}[X, R, V] \to \mathbb{C}[Y, S, W].
	$$
	This property is demonstrated in the proof of Lemma~\ref{lemma: K-theory of uniform covers} below.
	The following lemma shows that an approximable operator -- hence, in particular, of controlled propagation -- which satisfies conditions~$(2)$ and~$(3)$ of Definition~\ref{definition: uniform cover} defines an element of the multiplier algebra of $C^*_u(X) \otimes \mathbb{K}$.

	\begin{lemma}\label{lemma: multiplier algebra}
		Let $(X,d)$ be a uniformly locally finite metric space. Any approximable operator $S \in \mathcal{B}(\ell^2(X, H))$ that satisfies conditions~(2) and~(3) of Definition~\ref{definition: uniform cover} belongs to the multiplier algebra of $C^*_u(X) \otimes \mathbb{K}$.
	\end{lemma}
	\begin{proof}
		Since $C^*_u(X) \otimes \mathbb{K}$ is a non-degenerate $C^*$-subalgebra of $\mathcal{B}(\ell^2(X,H))$, it follows from \cite[II 7.3.5]{blackadar:2006zz} that its multiplier algebra can be identified with the $C^*$-subalgebra of $\mathcal{B}(\ell^2(X,H))$ given by
		\begin{equation*}
			\begin{split}
				\mathcal{A} := 
				\bigl\{ \, \,
				T \in \mathcal{B}(\ell^2(X,H)) \,\big| & \,
				T(C^*_u(X) \otimes \mathbb{K}) \subset C^*_u(X) \otimes \mathbb{K} \\
				\text{ and } &\,
				(C^*_u(X) \otimes \mathbb{K})T \subset C^*_u(X) \otimes \mathbb{K}
				\,\, \bigr\}.
			\end{split}
		\end{equation*}
		Let $S \in \mathcal{B}(\ell^2(X, H))$ be an approximable operator that satisfies conditions~(2) and~(3) of Definition~\ref{definition: uniform cover}. The $C^*$-algebra $C^*_u(X) \otimes \mathbb{K}$ is generated by operators of the form $b \otimes F$, where $b \in C^*_u(X)$ and $F$ is a finite-rank operator. Let $W \subset H$ be a subspace such that $S$ restricts to a map
		\[
			\ell^2(X,\operatorname{im}(F)) \longrightarrow \ell^2(X,W).
		\]
		Then the operator $S(b \otimes F)$ restricts to a map
		\[
			S(b \otimes F) \colon
			\ell^2(X,H) \longrightarrow \ell^2(X,W),
		\]
		and is approximable. Hence, $S(b \otimes F)$ belongs to $C^*_u(X) \otimes \mathbb{K}$. Since conditions~(2) and (3) of Definition~\ref{definition: uniform cover} are stable under taking adjoints, the operator $(b \otimes F)S$ also belongs to $C^*_u(X) \otimes \mathbb{K}$. It follows that $S \in \mathcal{A}$.
	\end{proof}
	
	It is clear from Definition~\ref{definition: uniform cover} that if $S$ is a uniform cover for a coarse map $f \colon (X,d) \to (Y,\partial)$, and $g \colon (X,d) \to (Y,\partial)$ is a coarse map close to $f$, then $S$ is also a uniform cover for $g$. 
	
	For the case of coarse equivalences, the existence of covering unitaries was established in~\cite{brodzki2007property}. The following lemma is a refinement of~\cite[Theorem~4]{brodzki2007property}, it establishes the existence of a uniform cover for any uniformly finite-to-one coarse map.

	\begin{lemma} \label{lemma: uniform covers existence}
		Let $(X,d)$ and $(Y,\partial)$ be uniformly locally finite metric spaces, let $f \colon (X,d) \to (Y,\partial)$ be a uniformly finite-to-one coarse map, and let $H$ be a separable Hilbert space. Then there exists a uniform cover
		\[
			S \colon \ell^2(X,H) \to \ell^2(Y,H)
		\]
		of $f$. Moreover, if $f$ is a coarse equivalence, $S$ can be chosen to be unitary.
	\end{lemma}
	\begin{proof}
		Without loss of generality, let $H = \ell^2(\mathbb{N})$. Since $f$ is uniformly finite-to-one, the quantity
		\[
			N := \sup_{y \in Y} \lvert f^{-1}(\{y\}) \rvert
		\]
		is finite. For $0 \le i \le N$, let $Y_i \subset Y$ denote the set of those $y \in Y$ whose preimage under $f$ consists of precisely $i$ distinct points. In this way, we obtain a partition
		\[
			Y = \bigsqcup_{i=0}^N Y_i.
		\]
		For each $1 \le i \le N$, fix a partition $\{A_k^i\}_{k=1}^i$ of $\mathbb{N}$ into $i$ pairwise disjoint subsets, each of which is in bijection with $\mathbb{N}$. Let $f_k^i \colon \mathbb{N} \to A_k^i$ be a bijection. For every $y \in \operatorname{im}(f)$, enumerate the finite set $f^{-1}(\{y\})$. Define an isometry $S$ on the canonical basis of $\ell^2(X) \otimes \ell^2(\mathbb{N})$ by setting
		\[
			S(\delta_{x_k} \otimes \delta_n) = \delta_y \otimes \delta_{f_k^i(n)},
		\]
		where $y \in Y_i$ for some $i \neq 0$ and $f(x_k) = y$. It is straightforward to verify that $\operatorname{supp}(S) = f$. Given $\zeta \in \ell^2(\mathbb{N})$, $y \in Y_i$, and $x_k \in f^{-1}(\{y\})$, one has
		\[
			S(\delta_{x_k} \otimes \zeta) = 
				\delta_y \otimes \zeta(k,i),
			\qquad
			\text{where}
			\quad 
			\zeta(k,i) := 
				\sum_{n \in \mathbb{N}} \langle \zeta, \delta_n \rangle 
				\,\delta_{f_k^i(n)}.
		\]
		Let $H_{\zeta}$ be the subspace of $\ell^2(\mathbb{N})$ generated by the vectors $\zeta(k,i)$ for $1 \le k \le i \le N$. This subspace is finite-dimensional, and
		\[
			S\bigl(\ell^2(X) \otimes \zeta\bigr) \subset \ell^2(Y) \otimes H_{\zeta}.
		\]
		Now let $H_0$ be any finite-dimensional subspace of $\ell^2(\mathbb{N})$, and let $H_1$ denote the subspace generated by $H_{\zeta}$ for all $\zeta \in H_0$. Then $H_1$ is finite-dimensional, and
		\[
			S\bigl(\ell^2(X) \otimes H_0\bigr) \subset \ell^2(Y) \otimes H_1.
		\]
		Analogously, for $\eta \in \ell^2(\mathbb{N})$ and $y_0 \in Y_i$ with $i \neq 0$, one has
		\[
			S^*(\delta_{y_0} \otimes \eta) = 
				\sum_{k=1}^i \delta_{x_k} \otimes \eta(k,i),
			\qquad 
			\text{where }
			\eta(k,i) := 
				\sum_{n \in N_k^i} \langle \eta, \delta_n \rangle 
				\, \delta_{(f_k^i)^{-1}(n)}.
		\]
		Moreover, for any $y \in Y_0$, one has $S^*(\delta_y \otimes \eta) = 0$. Let $H_{\eta}$ be the subspace of $\ell^2(\mathbb{N})$ generated by the vectors $\eta(k,i)$ for $1 \le k \le i \le N$. This subspace is finite-dimensional and
		\[
			S^*(\ell^2(Y) \otimes \eta) \subset \ell^2(X) \otimes H_{\eta}.
		\]
		For any finite-dimensional subspace $H_0 \subset \ell^2(\mathbb{N})$, let $H_1$ denote the subspace generated by $H_{\eta}$ for all $\eta \in H_0$. Then $H_1$ is finite-dimensional, and
		\[
			S^*(\ell^2(Y) \otimes H_0) \subset \ell^2(X) \otimes H_1.
		\]
		It follows that the constructed isometry $S$ satisfies the definition of a covering isometry for $f$. Note that if $f$ is surjective, then $S$ is surjective; in particular, $S$ is a unitary.

		Suppose that $f$ is a coarse equivalence. Then it can be decomposed into maps
		\[
			\begin{tikzcd}
				X \arrow{r}{f_0} & f(X) \arrow{r}{i} & Y,
			\end{tikzcd}
		\]
		where $f_0$ is a surjective coarse equivalence, and $i$ is the inclusion of a coarsely dense subset. Applying the construction of a covering isometry described above to the surjective coarse equivalence $f_0$, we obtain a unitary
		\[
			U_{f_0} \colon \ell^2(X) \otimes \ell^2(\mathbb{N}) \to \ell^2(f(X)) \otimes \ell^2(\mathbb{N}).
		\]
		Since $i \colon f(X) \to Y$ is the inclusion of a coarsely dense subset, there exists a left inverse $p \colon Y \to f(X)$, which is a surjective coarse equivalence. Applying the same construction to $p$ yields a covering unitary
		\[
			U_p \colon \ell^2(Y) \otimes \ell^2(\mathbb{N}) \to \ell^2(f(X)) \otimes \ell^2(\mathbb{N}).
		\]
		It is straightforward to verify that the unitary $U_p^* U_{f_0}$ satisfies Definition~\ref{definition: uniform cover}.
	\end{proof}
	
	Note that the proof of Lemma~\ref{lemma: uniform covers existence} involves a substantial number of choices; consequently, the resulting uniform cover is highly non-canonical. For instance, if $X$ is a singleton, then any isometry in $\mathcal{B}(H)$ covers $\operatorname{id}_{X}$. In the case of covering isometries for Roe algebras, this non-canonicity disappears at the level of $K$-theory. The following lemma shows that the same phenomenon holds for uniform covers.
	
	\begin{lemma} \label{lemma: K-theory of uniform covers}
		Let $(X,d)$ and $(Y,\partial)$ be uniformly locally finite metric spaces, let $f \colon (X,d) \to (Y,\partial)$ be a uniformly finite-to-one coarse map, and let $H$ be a separable Hilbert space. Let $S \colon \ell^2(X,H) \to \ell^2(Y,H)$ be a uniform cover of $f$. Then $\operatorname{Ad}_S$ restricts to a $*$-homomorphism
		\[
			\operatorname{Ad}_S \colon C_u^*(X) \otimes \mathbb{K}
				\longrightarrow C_u^*(Y) \otimes \mathbb{K}.
		\]
		Moreover, the induced map in $K$-theory does not depend on the choice of a uniform cover for $f$.

	\end{lemma}
	\begin{proof}
		Let $f \colon (X,d) \to (Y,\partial)$ be a uniformly finite-to-one coarse map, and let $S \colon \ell^2(X,H) \to \ell^2(Y,H)$ be a uniform cover of $f$. Fix $R \ge 0$ and a finite-dimensional subspace $H_0 \subset H$. Since $f$ is coarse, there exists $R_0 \ge 0$ such that for any $x,y \in X$ satisfying $d(x,y) \le R$, one has $d(f(x),f(y)) \le R_0$. By the definition of a uniform cover, there exists $R_1 \ge 0$ such that
		\[
			\|\mathbbm{1}_{y} S \mathbbm{1}_{x}\| \neq 0
				\;\Longrightarrow\;
				d(y,f(x)) \le R_1.
		\]
		In particular, for any operator $T \in \mathcal{B}(\ell^2(X,H))$ of propagation at most $R$, the operator $STS^*$ has propagation at most $R_0 + 2R_1$. Since $S$ satisfies condition~$(2)$ of Definition~\ref{definition: uniform cover}, there exists a finite-dimensional subspace $H_1 \subset H$ such that $S$ restricts to an isometry from $\ell^2(X,H_0)$ to $\ell^2(Y,H_1)$. It follows that for any $T \in \mathcal{B}(\ell^2(X, H_0))$ the operator $STS^*$ belongs to $\mathcal{B}(\ell^2(Y,H_1))$. In particular, $\operatorname{Ad}_S$ restricts to a map
		\[
			\operatorname{Ad}_S \colon
				\mathbb{C}[X,R,H_0] 
				\longrightarrow
				\mathbb{C}[Y,R_0 + 2R_1,H_1].
		\]
		By passing to closures, it follows that $\operatorname{Ad}_S$ restricts to a $*$-homomorphism between stabilisations of uniform Roe algebras. 
		
		For the $K$-theory part, let $S_1$ and $S_2$ be two uniform covers of $f$. For convenience, set $A = M_2(\mathbb{C}) \otimes C_u^*(Y) \otimes \mathbb{K}$. For $i = 1,2$, consider the maps
		\[
			\alpha_i \colon C_u^*(X) \otimes \mathbb{K} \longrightarrow A,
			\qquad
			\alpha_i(T) = e_{ii} \otimes \operatorname{Ad}_{S_i}(T).
		\]
		It suffices to show that $\alpha_1$ and $\alpha_2$ agree in $K$-theory. The projections $S_1 S_1^*$ and $S_2 S_2^*$, as well as the partial isometries $S_1 S_2^*$ and $S_2 S_1^*$, satisfy conditions of Lemma~\ref{lemma: multiplier algebra}; therefore, they belong to the multiplier algebra of $C_u^*(Y) \otimes \mathbb{K}$. Consider the unitary
		\[
			U =
				\begin{pmatrix}
					1 - S_1 S_1^* & S_1 S_2^* \\
					S_2 S_1^* & 1 - S_2 S_2^*
				\end{pmatrix}
		\]
		in the multiplier algebra of $A$. It follows that $\operatorname{Ad}_U$ induces the identity map on $K$-theory and intertwines $\alpha_1$ and $\alpha_2$. Therefore, $K_*(\alpha_1) = K_*(\alpha_2)$.
	\end{proof}
	
	In light of Lemma~\ref{lemma: K-theory of uniform covers}, for a uniformly finite-to-one coarse map $f$, we denote by $K_*(f)$ the map $K_*(\operatorname{Ad}_S)$ for some (any) uniform cover $S$. Since uniform covers depend only on the closeness class of a coarse map, the homomorphism $K_*(f)$ depends only on the closeness class of $f$. It follows that, for $i = 1,2$, there are well-defined functors
	\[
		K_i \colon \mathbf{Coarse} \longrightarrow \mathbf{Ab},
			\qquad
		(X,d) \longmapsto K_i\bigl(C_u^*(X)\bigr),
			\qquad
		f \longmapsto K_i(f),
	\]
	from the coarse category to the category of abelian groups. Another family of functors from the coarse category to abelian groups is given by uniformly finite homology (see Definition~\ref{definition: uf-homology}). Recall from Definition~\ref{definition: comparison map} that there is a map
	\[
		\alpha_0 \colon H_0^{\mathrm{uf}}(X;\mathbb{Z})
		\longrightarrow K_0\bigl(C_u^*(X)\bigr),
	\]
	called the $0$th comparison map. The following lemma asserts that $\alpha_0$ is a natural transformation from $0$th uniformly finite homology to $K_0$.
	
	\begin{lemma} \label{lemma: naturality of alpha0}
		Let $(X,d)$ and $(Y,\partial)$ be uniformly locally finite metric spaces, and denote by $\alpha_0^X$ and $\alpha_0^Y$ the associated $0$th comparison maps. Let $f \colon (X,d) \to (Y,\partial)$ be a uniformly finite-to-one coarse map. Then the following diagram commutes:
		\[
			\begin{tikzcd}
				H_0^{\mathrm{uf}}(X;\mathbb{Z})
					\arrow{r}{\alpha_0^X}
					\arrow[swap]{d}{H_0^{\mathrm{uf}}(f)} &
				K_0(C_u^*(X))
					\arrow{d}{K_0(f)}
				\\
				H_0^{\mathrm{uf}}(Y;\mathbb{Z})
					\arrow{r}{\alpha_0^Y} &
				K_0(C_u^*(Y)).
			\end{tikzcd}
		\]
	\end{lemma}
	\begin{proof}
		Suppose that $f \colon (X,d) \to (Y, \partial)$ is an injective coarse map. Define an isometry
		\[
			S \colon \ell^2(X) \to \ell^2(Y), 
			\qquad 
			S(\delta_x) := \delta_{f(x)}.
		\]
		By definition, $\widetilde{S} := S \otimes \operatorname{id}_H$ is a uniform cover of $f$; therefore, the maps $K_0(f)$ and $K_0(\operatorname{Ad}_S)$ coincide. Moreover, the $*$-homomorphism $\operatorname{Ad}_S$ restricts to a $*$-homomorphism between the canonical Cartan subalgebras $\ell^{\infty}(X)$ and $\ell^{\infty}(Y)$ of $C^*_u(X)$ and $C^*_u(Y)$, respectively. In particular, the following diagram commutes:
		\[
			\begin{tikzcd}
				K_0(\ell^{\infty}(X)) \arrow{r}{K_0(i_X)} \arrow[swap]{d}{K_0(\operatorname{Ad}_S)}
				& K_0(C^*_u(X)) \arrow{d}{K_0(f)} \\
				K_0(\ell^{\infty}(Y)) \arrow{r}{K_0(i_Y)}
				& K_0(C^*_u(Y)),
			\end{tikzcd}
		\]
		where $i_X$ and $i_Y$ denote the canonical inclusions. For a uniformly locally finite metric space $(Z,d_Z)$, let $\pi_Z$ be the canonical projection from $K_0(\ell^{\infty}(Z))$ onto $H_0^{\mathrm{uf}}(Z; \mathbb{Z})$. Since $f$ is injective, a direct computation shows that, for any subset $A \subseteq X$, one has $\operatorname{Ad}_S(\mathbbm{1}_A) = \mathbbm{1}_{f(A)}$ and $f_*(\mathbbm{1}_A) = \mathbbm{1}_{f(A)}$. The following diagram commutes:
		\[
			\begin{tikzcd}
				K_0(\ell^{\infty}(X)) \arrow{r}{\pi_X} \arrow[swap]{d}{K_0(\operatorname{Ad}_S)}
				& H_0^{\mathrm{uf}}(X; \mathbb{Z}) \arrow{d}{H_0^{\mathrm{uf}}(f)} \\
				K_0(\ell^{\infty}(Y)) \arrow{r}{\pi_Y}
				& H_0^{\mathrm{uf}}(Y; \mathbb{Z}).
			\end{tikzcd}
		\]
		Finally, since the map on $K_0$ induced by the inclusion $\ell^{\infty}(X) \hookrightarrow C^*_u(X)$ factors through the $0$th comparison map, the diagram appearing in the statement commutes.  
		
		For the general case, let $f \colon (X,d) \to (Y,\partial)$ is a uniformly finite-to-one coarse map, and set
		\[
			n := \sup_{y \in Y} \lvert f^{-1}(y) \rvert .
		\]
		Let $Y^{(n)}$ denote the $n$th doubling of $Y$, and let $j \colon Y \to Y^{(n)}$ be the inlcusion of $Y$ onto $Y \times \{1\} \subset Y^{(n)}$. The inclusion $j$ is a coarse equivalence. For every $y \in Y$ whose preimage has $i \neq 0$ points, fix an enumeration $x_1^y, \ldots, x_i^y$ of $f^{-1}(\{y\})$. Define a map
		\[
			\tilde{f} \colon X \to Y^{(n)}, \qquad x_i^{f(x)} \mapsto (f(x), i).
		\]
		It is straightforward to verify that $\tilde{f}$ is injective, and it is close to $j \circ f$. In particular, one has the equalities
		\[
			H_0^{\mathrm{uf}}(\tilde{f}) 
				= H_0^{\mathrm{uf}}(j_1) \circ H_0^{\mathrm{uf}}(f),
			\qquad
			K_0(\tilde{f})
				= K_0(j_1) \circ K_0(f).
		\]
		Since $\tilde{f}$ and $j$ are injective, the diagrams appearing in the statement commute for $\tilde{f}$ and $j$ by the first part of the proof. Since $j$ is a coarse equivalence, the maps $H_0^{\mathrm{uf}}(j)$ and $K_0(j)$ are isomorphisms. The following diagram commutes:
		\[
			\begin{tikzcd}
				H_0^{\mathrm{uf}}(X; \mathbb{Z})
  					\arrow[swap]{dd}{H_0^{\mathrm{uf}}(f)}
  					\arrow{rrr}{\alpha_0^X}
  					\arrow{dr}{H_0^{\mathrm{uf}}(\tilde{f})}
					&&&
				K_0(C^*_u(X))
  					\arrow[swap]{dl}{K_0(\tilde{f})}
  					\arrow{dd}{K_0(f)}
				\\
				& H_0^{\mathrm{uf}}(Y^{(n)}; \mathbb{Z})
    				\arrow{ld}{H_0^{\mathrm{uf}}(j_1)^{-1}}
    				\arrow{r}{\alpha_0^{Y^{(n)}}}
				& K_0(C^*_u(Y^{(n)}))
    				\arrow[swap]{dr}{K_0(j_1)^{-1}}
				\\
				H_0^{\mathrm{uf}}(Y; \mathbb{Z})
 	 				\arrow{rrr}{\alpha_0^Y}
					&&&
				K_0(C^*_u(Y)).
			\end{tikzcd}
		\]
		In particular, the diagram appearing in the statement commutes.
	\end{proof}

	There is a more conceptual explanation of Lemma~\ref{lemma: naturality of alpha0}, which we briefly outline without entering into technical details. Recall from Section~\ref{subsection: coarse geometry} that to each uniformly locally finite metric space $(X,d)$ one can associate a principal ample groupoid $\mathcal{G}(X,d)$, called the \emph{coarse groupoid} of $(X,d)$. Following Proposition~2.3 of~\cite{skandalis2002coarse}, any uniformly finite-to-one coarse map
	\[
		f \colon (X,d) \longrightarrow (Y,\partial)
	\]
	induces a uniformly locally finite coarse structure $\mathcal{E}_f$ on the disjoint union $X \sqcup Y$, whose restriction to $Y$ coincides with the original coarse structure of $Y$, and whose restriction to $X$ contains the original coarse structure of $X$. Consequently, at the level of coarse groupoids, one obtains \'etale groupoids homomorphisms
	\[
		\begin{tikzcd}
			\mathcal{G}(X,d) \arrow{r}{i_X} &
			\mathcal{G}(X \sqcup Y, \mathcal{E}_f) &
			\arrow[swap]{l}{i_Y} \mathcal{G}(Y,\partial),
		\end{tikzcd}
	\]
	where the right-hand arrow is a Morita equivalence of groupoids. Passing to groupoid homology or to the $K$-theory of reduced groupoid $C^*$-algebras yields group homomorphisms
	\[
		F(f) \colon F(\mathcal{G}(X,d)) \longrightarrow F(\mathcal{G}(Y,\partial)), 
		\qquad
		F(f) = F(i_Y)^{-1} \circ F(i_X),
	\]
	where $F$ denotes either groupoid homology or the $K$-theory of reduced groupoid $C^*$-algebras. The $0$th comparison map for ample groupoids is a natural transformation; therefore, the following diagram commutes:
	\[
		\begin{tikzcd}
			H_0(\mathcal{G}(X,d);\mathbb{Z})
				\arrow{r}{\alpha_0}
				\arrow[swap]{d}{H_0(f)} &
			K_0(C_r^*(\mathcal{G}(X,d)))
				\arrow{d}{K_0(f)}
			\\
			H_0(\mathcal{G}(Y,\partial);\mathbb{Z})
				\arrow{r}{\alpha_0} &
			K_0(C_r^*(\mathcal{G}(Y,\partial))).
		\end{tikzcd}
	\]
	Moreover, by~\cite[Theorem G]{bonicke2023dynamic}, there is a natural isomorphism between the homology of coarse groupoids and uniformly finite homology. As the uniform Roe algebra is naturally isomorphic to the reduced groupoid $C^*$-algebra of the associated coarse groupoid, the commutativity of the diagram in Lemma~\ref{lemma: naturality of alpha0} follows.

	\section{Rigidity of uniform Roe algebras versus injectivity of $\alpha_0$} \label{section: rigidity vs injectivity}
	
	In this section, we establish a close connection between the injectivity of the $0$th comparison map for coarse groupoids and the bijective rigidity problem. In contrast to arguments involving a single space, our approach necessarily requires the consideration of subspaces and doublings.

	\begin{definition} \label{definition: admissible class}
		A collection $\mathcal{C}$ of uniformly locally finite metric spaces is called \emph{admissible} if it is closed under taking arbitrary doublings and subspaces.
	\end{definition}
	
	For example, let $\mathcal{C}(\mathrm{A})$ denote the collection of all uniformly locally finite metric spaces that satisfy Property~A, and let $\mathcal{C}(\mathrm{CE})$ denote the collection of all uniformly locally finite metric spaces that admit a coarse embedding into a Hilbert space. Since both properties are preserved under taking subspaces and under passage to coarsely equivalent spaces, the collections $\mathcal{C}(\mathrm{A})$ and $\mathcal{C}(\mathrm{CE})$ are admissible. Clearly, the collection of all uniformly locally finite metric spaces is also admissible. The following lemma is the main result of this section.
	
	\begin{lemma}[Lemma~A]\label{Body: lemma: A}
		Let $\mathcal{C}$ be an admissible collection of uniformly locally finite metric spaces. The following statements are equivalent:
		\begin{enumerate}
			\item For every uniformly locally finite metric space $(X,d)$ in $\mathcal{C}$, the $0$th comparison map $\alpha_0 \colon H^{\mathrm{uf}}_0(X; \mathbb{Z}) \to K_0\bigl(C^*_u(X)\bigr)$ is injective;
			\item For any uniformly locally finite metric spaces $(X,d)$ and $(Y,\partial)$ in $\mathcal{C}$ such that $C^*_u(Y) \cong C^*_u(X)$, the coarse equivalence given by Theorem~\ref{theorem: rigidity theorem} is close to a bijective coarse equivalence.
		\end{enumerate}
	\end{lemma}
	
	Note that, for ample groupoids, the injectivity of $\alpha_0$ is invariant under Morita equivalence; hence, the injectivity of $\alpha_0$ is preserved by passing to coarsely equivalent spaces. In particular, the $0$th comparison map for $(X,d)$ is injective if and only if it is injective for any (some) doubling of $(X,d)$. On the other hand, it is not evident whether injectivity passes to subspaces. Before proceeding with the proof of Lemma \ref{Body: lemma: A}, we need to establish a connection between the $K$-theory of an isomorphism $\Phi \colon C^*_u(X) \to C^*_u(Y)$, and the $K$-theory of the coarse equivalence $f_{\Phi}$ induced by Theorem~\ref{theorem: rigidity theorem}.

	\begin{proposition}\label{proposition: HK-square for isomorphisms of uniform Roe algebras}
		Let $(X,d)$ and $(Y,\partial)$ be uniformly locally finite metric spaces, and let $\Phi \colon C^*_u(X) \to C^*_u(Y)$ be an isomorphism. Denote by $f_{\Phi} \colon (X,d) \to (Y,\partial)$ the coarse equivalence provided by Theorem~\ref{theorem: rigidity theorem}. Then the induced maps in $K$-theory, $K_*(f_{\Phi})$ and $K_*(\Phi)$, coincide.
	\end{proposition}
	\begin{proof}
		Let $U$ be a uniform cover for $f_{\Phi}$. By~\cite[Lemma 3.1]{vspakula2013rigidity}, any isomorphism between uniform Roe algebras is spatially implemented; hence, there exists a unitary operator $V \colon \ell^2(X) \to \ell^2(Y)$ such that $\Phi = \operatorname{Ad}_V$. Consider the unitary operator
		\[
			\widetilde{U} := U^*(V \otimes \operatorname{id}_H).
		\]
		Note that $\widetilde{U}$ satisfies conditions~(2) and~(3) of Definition~\ref{definition: uniform cover}. By Theorem \ref{theorem: norm limit of unitaries supported on f_Phi}, the unitary $V \otimes \operatorname{id}_H$ is a norm limit of a net $\{T_\lambda\}_\lambda$ of operators supported on $f_{\Phi}$. It follows that $\widetilde{U}$ is a norm limit of operators supported on $\operatorname{id}_X$; in particular, $\widetilde{U}$ is approximable. By Lemma~\ref{lemma: multiplier algebra}, the unitary $\widetilde{U}$ belongs to the multiplier algebra of $C^*_u(X) \otimes \mathbb{K}$ and therefore induces the identity map in $K$-theory. Consequently,
		\[
			\operatorname{id}
				= K_*(\operatorname{Ad}_{\widetilde{U}})
				= K_*(\Phi)^{-1} \circ K_*(f_{\Phi}),
		\]
		and hence the maps $K_*(\Phi)$ and $K_*(f_{\Phi})$ coincide.
	\end{proof}
	\begin{remark}
		In the proof of Proposition~\ref{proposition: HK-square for isomorphisms of uniform Roe algebras}, we showed that there exists a unitary $\tilde{U}$ in the multiplier algebra of $C^*_u(X) \otimes \mathbb{K}$ which intertwines the $*$-isomorphisms $\Phi \otimes \operatorname{id}_H$ and $\operatorname{Ad}_U$, where $U$ is a uniform cover of $f_{\Phi}$. Recall from~\cite[Theorem 2.5]{mingo1987K} that for any unital $C^*$-algebra $A$, the unitary group of the multiplier algebra of $A \otimes \mathbb{K}$ is contractible. Consequently, there exists a continuous path of unitaries $u_t \in \mathcal{M}(C^*_u(X) \otimes \mathbb{K})$ such that $u_0 = \tilde{U}$ and $u_1 = 1$. In particular, the $*$-isomorphisms $\Phi \otimes \operatorname{id}_H$ and $\operatorname{Ad}_U$ are homotopic.
	\end{remark}
	
	By combining Proposition~\ref{proposition: HK-square for isomorphisms of uniform Roe algebras} and Lemma~\ref{lemma: naturality of alpha0}, we deduce that for any $*$-isomorphism of uniform Roe algebras $\Phi \colon C^*_u(X) \to C^*_u(Y)$ the following diagram commutes:
	\begin{equation} \label{equation: HK-square}
		\begin{tikzcd}
			H_0^{\text{uf}}(X; \mathbb{Z})
				\arrow{r}{\alpha_0^X}
				\arrow[swap]{d}{H_0^{\text{uf}}(f_{\Phi})} &
			K_0(C^*_u(X))
				\arrow{d}{K_0(\Phi)}
			\\
			H_0^{\text{uf}}(Y; \mathbb{Z})
				\arrow{r}{\alpha_0^Y} &
			K_0(C^*_u(Y)),
		\end{tikzcd}
	\end{equation}
	where $f_{\Phi}$ denotes the coarse equivalence induced by $\Phi$ via Theorem~\ref{theorem: rigidity theorem}.
	
	\begin{proof}[Proof of Lemma \ref{Body: lemma: A}]
		Let $\mathcal{C}$ be an admissible collection of uniformly locally finite metric spaces such that, for every $(X,d) \in \mathcal{C}$, the comparison map
		\[
			\alpha_0^X \colon H_0^{\text{uf}}(X; \mathbb{Z}) \to K_0(C^*_u(X))
		\]
		is injective. Suppose that $(X,d)$ and $(Y,\partial)$ belong to $\mathcal{C}$ and $\Phi \colon C^*_u(X) \to C^*_u(Y)$ is a $*$-isomorphism. By the commutativity of~\eqref{equation: HK-square}, we obtain
		\[
			\alpha_0^Y \circ H_0^{\text{uf}}(f_{\Phi})([\mathbbm{1}_X])
				=
			K_0(\Phi) \circ \alpha_0^X([\mathbbm{1}_X]),
		\]
		and the right-hand side is exactly the $K_0$-class of $\mathbbm{1}_Y$. Since $\alpha_0^Y$ is injective, it follows that $H_0^{\text{uf}}(f_{\Phi})$ preserves the fundamental class. By Theorem~\ref{theorem: Block-Weinberger-Whyte}, there exists a bijective coarse equivalence that is close to $f_{\Phi}$. Consequently, the coarse equivalence given by Theorem~\ref{theorem: rigidity theorem} is close to a bijective coarse equivalence.
		
		Conversely, let $\mathcal{C}$ be an admissible collection of uniformly locally finite metric spaces such that for any $(X,d)$ and $(Y,\partial)$ in $\mathcal{C}$ satisfing $C^*_u(Y) \cong C^*_u(X)$, the coarse equivalence given by Theorem~\ref{theorem: rigidity theorem} is close to a bijective coarse equivalence. Let $(X,d) \in \mathcal{C}$, and let $\beta \in \ker(\alpha_0^X)$. By passing, if necessary, to a doubling of $(X,d)$, for some $n \ge 1$ the class $\beta$ is given by
		\[
			\beta = [\mathbbm{1}_B] - [\mathbbm{1}_A],
		\]
		for some subsets $A,B \subset X^{(n)}$. Since $\beta$ lies in the kernel of $\alpha_0^X$, the projections $\mathbbm{1}_A$ and $\mathbbm{1}_B$ define the same class in $K_0(C^*_u(X^{(n)}))$. Consequently, there exist $\ell,k \in \mathbb{N}$ and a partial isometry $s$ in $C^*_u(X^{(n)}) \otimes M_{\ell+k+1}(\mathbb{C})$ such that
		\[
			\mathbbm{1}_A \oplus \mathbbm{1}_{X^{(n)}}^{\oplus \ell} \oplus 0^{\oplus k}
			= ss^* \sim s^*s = 
			\mathbbm{1}_B \oplus \mathbbm{1}_{X^{(n)}}^{\oplus \ell} \oplus 0^{\oplus k}.
		\]
		It follows that $\Phi := \operatorname{Ad}_s$ induces an isomorphism between the uniform Roe algebras of $A \sqcup X^{(n+\ell)}$ and $B \sqcup X^{(n+\ell)}$. By Theorem~\ref{theorem: rigidity theorem}, the isomorphism $\Phi$ gives rise to a coarse equivalence
		\[
			f_{\Phi} \colon A \sqcup X^{(n+\ell)} \longrightarrow B \sqcup X^{(n+\ell)}.
		\]
		Recall that, up to closeness, for some (equivalently, any) $\delta \in (0,1)$, there exist constants $R, P \ge 0$ such that the relation $f_{\Phi}$ has the following form:
		\[
			f_{\Phi}
				=
			\bigcup \{C \times D \mid \operatorname{diam}(C) \le R,\ \operatorname{diam}(D) \le P,\ \|\mathbbm{1}_C s \mathbbm{1}_D\| > \delta\}.
		\]
		Since $s$ is approximable, for every $\varepsilon>0$ there exists a controlled propagation operator $t$ such that $\|s-t\|<\varepsilon$. Choose $\varepsilon>0$ with $\varepsilon<\delta$. If $C,D \subset X$ satisfy $\|\mathbbm{1}_C s \mathbbm{1}_D\|>\delta$, then
		\[
			\delta < \|\mathbbm{1}_C s \mathbbm{1}_D\| \le \|\mathbbm{1}_C (s-t) \mathbbm{1}_D\| + \|\mathbbm{1}_C t \mathbbm{1}_D\| < \varepsilon + \|\mathbbm{1}_C t \mathbbm{1}_D\|.
		\]
		In particular, $\|\mathbbm{1}_C t \mathbbm{1}_D\|>\delta-\varepsilon$. Consequently, we obtain a chain of inclusions
		\[
			\begin{split}
				f_{\Phi} = \;&
					\bigcup \{C \times D \mid \operatorname{diam}(C)\le R,\ \operatorname{diam}(D)\le P,\ \|\mathbbm{1}_C s \mathbbm{1}_D\|>\delta\}
				\\ \subset\;&
					\bigcup \{C \times D \mid \operatorname{diam}(C)\le R,\ \operatorname{diam}(D)\le P,\ \|\mathbbm{1}_C t \mathbbm{1}_D\|>\delta-\varepsilon\}
				\\ \subset\;&
					\bigcup \{C \times D \mid \operatorname{diam}(C)\le R,\ \operatorname{diam}(D)\le P,\ \|\mathbbm{1}_C t \mathbbm{1}_D\|\neq 0\}
				\\ \subset\;&
					E_R \circ \operatorname{supp}(t) \circ E_P .
			\end{split}
		\]
		Since $t$ has controlled propagation, its support is an entourage. As entourages are closed under composition, the latter set is also an entourage. In particular, the relation $f_{\Phi}$ is contained in an entourage; therefore, it is close to the identity map on $X$. By assumption, $f_{\Phi}$ can be chosen to be bijective; in particular, $f_{\Phi}$ is a partial translation. Hence the indicator functions of $A \sqcup X^{(n+\ell)}$ and $B \sqcup X^{(n+\ell)}$ define the same class in $H_0^{\mathrm{uf}}(X; \mathbb{Z})$. By substracting the homology class of the indicator function of $X^{(n+\ell)}$, we conclude that $[\mathbbm{1}_A] = [\mathbbm{1}_B]$, and therefore $\beta = 0$. This shows that $\alpha_0^X$ is injective.
	\end{proof}
	
	Note that the proof of the first implication does not rely on the assumption that $\mathcal{C}$ is an admissible collection. Indeed, the argument does not involve passing to subspaces or to doublings of the underlying metric spaces. Consequently, the conclusion of the first implication holds for a single space.
	
	\begin{corollary} \label{corollary: injectivity => bij.rigidity}
		Let $(X,d)$ be a uniformly locally finite metric space for which the $0$th comparison map $\alpha_0^X$ is injective. Then, for any uniformly locally finite metric space $(Y,\partial)$ whose uniform Roe algebra is isomorphic to $C^*_u(X)$, the coarse equivalence provided by Theorem~\ref{theorem: rigidity theorem} is close to a bijective coarse equivalence.
	\end{corollary}
	\begin{proof}
		Let $(X,d)$ and $(Y,\partial)$ be uniformly locally finite metric spaces, and let $\Phi \colon C^*_u(X) \to C^*_u(Y)$ be an isomorphism. Denote by $f_{\Phi}$ the coarse equivalence provided by Theorem~\ref{theorem: rigidity theorem}. Assume that the $0$th comparison map $\alpha_0^X$ is injective. By Lemma~\ref{lemma: naturality of alpha0}, since $H_0^{\mathrm{uf}}(f_{\Phi})$ and $K_0(f_{\Phi})$ are isomorphisms, it follows that the $0$th comparison map $\alpha_0^Y$ is also injective. Consequently, by the proof of Lemma~\ref{Body: lemma: A}, the coarse equivalence $f_{\Phi}$ is close to a bijective coarse equivalence.
	\end{proof}
	
	\begin{remark} \label{remark on rational injectivity}
		In light of Theorem~\ref{theorem: injectivity/isomorphism of extension of scalars}, the injectivity of the $0$th comparison map is equivalent to the injectivity of the rational comparison map
		\[
			\alpha_0^X \otimes \operatorname{id}_{\mathbb{Q}} \colon H_0^{\mathrm{uf}}(X; \mathbb{Z}) \otimes_{\mathbb{Z}} \mathbb{Q} \longrightarrow K_0(C^*_u(X)) \otimes_{\mathbb{Z}} \mathbb{Q}.
		\]
		Indeed, since $\mathbb{Q}$ is a torsion-free abelian group, it is projective; consequently, the functor $\otimes_{\mathbb{Z}} \mathbb{Q}$ preserves injective group homomorphisms. In particular, if $\alpha_0^X$ is injective, then so is $\alpha_0^X \otimes \operatorname{id}_{\mathbb{Q}}$. Conversely, for any torsion-free abelian group $A$, the canonical map
		\[
			i_A \colon A \longrightarrow A \otimes_{\mathbb{Z}} \mathbb{Q}, \qquad a \longmapsto a \otimes 1,
		\]
		is injective. By Theorem~\ref{theorem: injectivity/isomorphism of extension of scalars}, for any uniformly locally finite metric space $(X,d)$, the group $H_0^{\mathrm{uf}}(X; \mathbb{Z})$ is torsion-free. Consider the following commutative diagram:
		\[
			\begin{tikzcd}
				H_0^{\mathrm{uf}}(X; \mathbb{Z}) \otimes_{\mathbb{Z}} \mathbb{Q}
				\arrow{rr}{\alpha_0^X \otimes \operatorname{id}_{\mathbb{Q}}}
				& &
				K_0(C^*_u(X)) \otimes_{\mathbb{Z}} \mathbb{Q}
				\\
				H_0^{\mathrm{uf}}(X; \mathbb{Z})
				\arrow{rr}{\alpha_0^X} \arrow{u}{i} & &
				K_0(C^*_u(X)). \arrow{u}{i}
			\end{tikzcd}
		\]
		It follows that injectivity of $\alpha_0^X \otimes \operatorname{id}_{\mathbb{Q}}$ implies injectivity of $\alpha_0^X$. Consequently, in Lemma~\ref{Body: lemma: A} and its corollaries, it suffices to assume rational injectivity of the $0$th comparison map.
	\end{remark}

	\section{Consequences} \label{section: consequences}
	
	In \cite[Theorem E]{white2020cartan} it is shown that the class $\mathcal{C}(\mathrm{A})$ of uniformly locally finite metric spaces satisfying property~A (equivalently, the class of amenable coarse groupoids) satisfies bijective rigidity. By applying Lemma~\ref{Body: lemma: A}, we obtain that the $0$th comparison map is injective for all uniformly locally finite spaces that satisfy Property~A. The authors of~\cite{bonicke2023dynamic} (see Corollary C) proved that for any ample, principal, second-countable, locally compact, Hausdorff groupoid $\mathcal{G}$ with dynamic asymptotic dimension at most $2$, such that $H_2(\mathcal{G}; \mathbb{Z})$ is finitely generated, there are isomorphisms
	\[
		K_0(C^*_r(\mathcal{G})) \cong
		H_0(\mathcal{G};\mathbb{Z}) \oplus H_2(\mathcal{G};\mathbb{Z}),
		\qquad
		K_1(C^*_r(\mathcal{G})) \cong H_1(\mathcal{G};\mathbb{Z}).
	\]
	They also demonstrated that the result applies to coarse groupoids. For coarse groupoids, the dynamical asymptotic dimension coincides with the asymptotic dimension of the underlying metric space. Consequently, the above isomorphisms hold for all uniformly locally finite metric spaces of asymptotic dimension at most $2$ with finitely generated $H_2^{\text{uf}}(X; \mathbb{Z})$. Since every uniformly locally finite metric space of finite asymptotic dimension has property~A, our results may be viewed as a partial generalisation of the results of \cite{bonicke2023dynamic}. Recently, independently of this work, Vignati showed in~\cite{avignati} that bijective rigidity holds for all uniformly locally finite metric spaces (see Theorem~\ref{Intro: theorem: bij.rigidity}). Moreover, the bijective coarse equivalence constructed therein is close to the coarse equivalence provided by Theorem~\ref{theorem: rigidity theorem}. We obtain the following theorem.
	
	\begin{theorem}[Theorem B] \label{Body: theorem: B}
		Let $(X,d)$ be a uniformly locally finite metric space, and let $\mathcal{G}$ denote its coarse groupoid. Then the comparison map
		\[
			\alpha_0 \colon H_0(\mathcal{G}; \mathbb{Z}) \to K_0(C^*_r(\mathcal{G}))
		\]
		is injective. Moreover, if $(X,d)$ is coarsely connected, then $\alpha_0$ is split injective.
	\end{theorem}
	\begin{proof}
		By~\cite{avignati}, the second assertion of Lemma~\ref{Body: lemma: A} holds; therefore, for any uniformly locally finite metric space $(X,d)$, the comparison map $\alpha_0$ is injective. Suppose now that $(X,d)$ is coarsely connected. By Theorem~\ref{theorem: injectivity/isomorphism of extension of scalars}, the canonical map
		\[
			H_0^{\mathrm{uf}}(X; \mathbb{Z}) \longrightarrow H_0^{\mathrm{uf}}(X; \mathbb{R})
		\]
		is an isomorphism. In particular, $H_0^{\mathrm{uf}}(X; \mathbb{Z})$ carries the structure of a real vector space. It follows that any injective homomorphism from $H_0^{\mathrm{uf}}(X; \mathbb{Z})$ splits.
	\end{proof}
	
	Let $G$ be a finitely generated group equipped with the word-length metric. By~\cite[Proposition~A.10]{diana2017}, its $0$th uniformly finite homology group canonically identifies with the group homology of $G$ with coefficients in $\ell^{\infty}(G, \mathbb{R})$. Since Cayley graphs of finitely generated groups are coarsely connected, we obtain the following application.
	
	\begin{corollary}[Corollary C] \label{Body: corollary: C}
		Let $G$ be a finitely generated group equipped with the word-length metric. The canonical map
		\[
			\alpha_0^G \colon H_0\bigl(G; \ell^{\infty}(G,\mathbb{R})\bigr)
				\longrightarrow
			K_0\bigl(\ell^{\infty}(G) \rtimes_r G\bigr)
		\]
		is split-injective.
	\end{corollary}
	
	Corollary~\ref{Body: corollary: C} relies on the fact that real vector spaces are injective objects in the category of abelian groups. In general, however, $H_0^{\mathrm{uf}}(X; \mathbb{Z})$ need not be injective. For example, if $X$ is finite, then $H_0^{\mathrm{uf}}(X; \mathbb{Z})$ is isomorphic to $\mathbb{Z}$. As another example, consider the subset $\mathbb{N}_s := \{n^2 \mid n \in \mathbb{N}\}$ of $\mathbb{N}$ equipped with the induced metric. This space is sometimes referred to as the \emph{sparse natural numbers}. One readily verifies that
	\[
		H_0^{\mathrm{uf}}(\mathbb{N}_s; \mathbb{Z})
		\cong
		\ell^{\infty}(\mathbb{N}, \mathbb{Z}) / \{f \in c_0(\mathbb{N}, \mathbb{Z}) \mid \sum_{n \in \mathbb{N}} f(n) = 0\}.
	\]
	The latter group is not divisible and therefore is not injective. Note that, for any abelian group $A$, the group $A \otimes_{\mathbb{Z}} \mathbb{Q}$ is a $\mathbb{Q}$-vector space and therefore injective. It is an immediate consequence of Theorem~\ref{Body: theorem: B} and Remark~\ref{remark on rational injectivity} that the rational comparison map
	\[
		\alpha_0 \otimes \operatorname{id}_{\mathbb{Q}} \colon H_0(\mathcal{G}; \mathbb{Z}) \otimes_{\mathbb{Z}} \mathbb{Q} \to K_0(C^*_r(\mathcal{G})) \otimes_{\mathbb{Z}} \mathbb{Q}
	\]
	is split-injective. In \cite{proietti2025cherncharactertorsionfreeample}, the authors show that for any second-countable, locally compact, Hausdorff, ample groupoid $\mathcal{G}$ with torsion-free stabilisers, there is a group isomorphism
	\begin{equation} \label{equation: rational HK-conjecture}
		K_*^{\mathrm{top}}(\mathcal{G}, C_0(\mathcal{G}^{(0)})) \otimes_{\mathbb{Z}} \mathbb{Q} \longrightarrow \bigoplus_{k \ge 0} H_{*+2k}(\mathcal{G}; \mathbb{Z}) \otimes_{\mathbb{Z}} \mathbb{Q},
	\end{equation}
	where the left-hand side denotes the topological $K$-theory of $\mathcal{G}$. In particular, assuming the rational injectivity of the Baum--Connes assembly map, one recovers the rational split-injectivity of the comparison map. As coarse groupoids are, in general, not second-countable, one has to provide an additional limiting argument akin to the one used in \cite[Section 5.3]{bonicke2023dynamic} to apply the machinery of \cite{proietti2025cherncharactertorsionfreeample}. This might serve as an alternative proof of the split-injectivity of the rational comparison map for coarse groupoids whose Baum--Connes assembly map is rationally injective. For example, by the results of \cite{skandalis2002coarse, tu1999conjecture}, the class $\mathcal{C}(\mathrm{CE})$ of uniformly locally finite metric spaces admitting a coarse embedding into a Hilbert space satisfies the Baum–Connes conjecture with coefficients; in particular, for these spaces, the Baum-Connes assembly map is rationally injective.
	
	\begin{remark} \label{remark: Baum-Connes}
		Let $\mathcal{G}$ be a principal, Hausdorff, ample groupoid with compact unit space and let
		$$
			\mu \colon K^{\mathrm{top}}(\mathcal{G}; C(\mathcal{G}^{(0)})) \to K_0(C^*_r(\mathcal{G}))
		$$
		denote the Baum--Connes assembly map for $\mathcal{G}$ with trivial coefficents. It can be shown that there exists a natural group homomorphism
		$$
			\alpha^0 \colon H_0(\mathcal{G}; \mathbb{Z}) \to K^{\mathrm{top}}(\mathcal{G}; C(\mathcal{G}^{(0)})),
		$$
		such that the $0$th comparison map $\alpha_0$ factors through $\alpha^0$ and the Baum-Connes map. In the case of coarse groupoids, Theorem~\ref{Body: theorem: B} establishes an unconditional low-dimensional injectivity of the Baum-Connes map. We shall return to this perspective in our future work.
	\end{remark}
	
	\bibliographystyle{plain}
	\bibliography{RURAAICM.bib}
	
\end{document}